\documentclass{amsart}

\usepackage{hyperref}

\usepackage{amsmath}%
\usepackage{amsfonts}%
\usepackage{amssymb}%
\usepackage{graphicx}
\usepackage{subfigure}
\usepackage{color}
\usepackage{dsfont}
\usepackage{bbm}
\usepackage{cancel}
\usepackage{verbatim}

\newcommand{\domain}{\Lambda}

\newcommand{\Ol}{\Omega_{1}}
\newcommand{\Onl}{\Omega_{2}}
\newcommand{\Gl}{{\Gamma_1}}
\newcommand{\Gnl}{{\Gamma_2}}
\renewcommand{\sl}{{\sigma_\ell}}
\newcommand{\snl}{{\sigma_{n\ell}}}

\def\calT{\mathcal{T}}
\def\calA{\mathcal{A}}
\def\calD{\mathcal{D}}

\def\calH{\mathcal{H}}
\def\calS{\mathcal{S}}
\newcommand{\Th}{{\mathcal {T}_h}}
\def\vn{{\mathbf n}}

\def\Om{\Omega}

\def\R{{\mathbb R}}

\def\rn{{{\mathbb R}^n}}

\newcommand{\phii}{{\varphi}}

\newcommand{\Ns}{\mathcal{N}_s}
\newcommand{\Ms}{\mathcal{M}_s}
\newcommand{\Msnl}{\mathcal{M}_{s,\snl}}
\newcommand{\Nsnl}{\mathcal{N}_{s,\snl}}

\newcommand{\x}{\texttt{x}}

\renewcommand{\div}{{\rm div}\,}

\newtheorem{theorem}{Theorem}
\theoremstyle{plain}

\newtheorem{definition}{Definition}
\newtheorem{lemma}{Lemma}
\newtheorem{proposition}{Proposition}
\newtheorem{remark}{Remark}

\numberwithin{equation}{section}
\numberwithin{table}{section}
\numberwithin{figure}{section}
\numberwithin{theorem}{section}
\numberwithin{corollary}{section}
\numberwithin{definition}{section}
\numberwithin{lemma}{section}
\numberwithin{proposition}{section}
\numberwithin{remark}{section} 

\begin{document}
\title[]{On some coupled local and nonlocal diffusion models}
\author[J.P.~Borthagaray]{Juan Pablo~Borthagaray}
\address[J.P.~Borthagaray]{Instituto de Matem\'atica y Estad\'istica Rafael Laguardia, Universidad de la Rep\'ublica \\ 565 Julio Herrera y Reissig, 11300 Montevideo, Uruguay}
\email{jpborthagaray@fing.edu.uy}
\thanks{The authors have been supported in part by MATHAmSud project 22-MATH-04. JPB has been supported in part by Fondo Clemente Estable grant 172393.}

\author[P.~Ciarlet~Jr.]{Patrick~Ciarlet~Jr.}
\address[P.~Ciarlet~Jr.]{POEMS, CNRS, Inria, ENSTA, Institut Polytechnique de Paris\\ 828 Boulevard des Mar\'{e}chaux, 91120 Palaiseau, France}
\email{patrick.ciarlet@ensta.fr} 
\begin{abstract}
We study problems in which a local model is coupled with a nonlocal one.  We propose two energies: both of them are based on the same classical weighted $H^1$-semi norm to model the local part, while two different weighted $H^s$-semi norms, with $s \in (0,1)$,
are used to model the nonlocal part. The corresponding strong formulations are derived. In doing so, one needs to develop some technical tools, such as suitable integration by parts formulas for operators with variable diffusivity, and one also needs to study the mapping properties of the Neumann operators that arise.
In contrast to problems coupling purely local models, in which one requires transmission conditions on the interface between the subdomains, the presence of a nonlocal operator may give rise to nonlocal fluxes. These nonlocal fluxes may enter the problem as a source term, thereby changing its structure. Finally, we focus on a specific problem, that we consider most relevant, and study regularity of solutions and finite element discretizations. We provide numerical experiments to illustrate the most salient features of the models.
\end{abstract}

\maketitle 

\section{Introduction and problem setting}

Our goal is to study problems in which local and nonlocal diffusion scalar models are coupled, from both a theoretical and a numerical point of view. To that aim, we rely on energy-based models, which involve a weighted $H^1$-semi norm (squared) to model the local part, and a weighted $H^s$-semi norm (squared) to model the nonlocal part, where $s\in(0,1)$. We observe that similar models have already been proposed by Acosta {\em et al.} \cite{AcBeRo22, AcBeRo21}, who addressed their well-posedness. These results serve as a foundation for our analysis; see Sections~\S\S\ref{subsec-Motivation}--\ref{sec:l-nl-coupling}.
In this manuscript, we focus on the transmission conditions between the two models and,  in particular, we investigate whether they can be reformulated using classical transmission conditions. It turns out this is not always the case. Also, different behaviors are to be expected, depending on the chosen nonlocal model, that is, on the value of $s$.
We also study the a priori regularity of the solutions, which depends on the geometry and in particular on the shape of the interface between the local and nonlocal models. We refer to \cite{Kr-thesis, Kr15, GiStSt19, EsGiSt20} for related regularity studies. This regularity allows us to carry out the numerical analysis after discretization by finite element techniques of the coupled models. A natural idea is to use the same discretization for both the local and the nonlocal models, so that the actual implementation is straightforward. Using classical tools such as mesh refinement techniques \cite{AcBo17}, one can then recover optimal convergence rates. 

Other coupling strategies between local and nonlocal models have already been proposed in the literature, in particular coupling classical elasticity models with peridynamic formulations.
Special attention has been given to nonlocal models with integrable kernels, and in this situation one needs to address the local limit in which the so-called interaction horizon vanishes. We refer to \cite{DuLiLuTi18,CaDEBoGu20,DiPr22,DELiSeTiYu22} and the references therein for further details.

We next describe the domain configuration. Let $\Omega \subset \rn$ be a bounded Lipschitz domain, that is partitioned into two disjoint Lipschitz subdomains $\Ol$ and $\Onl$.
We define the interface as $\Sigma := int(\overline\Ol \cap \overline\Onl) $ and, for $k = 1,2$, we set $\Gamma_k=int(\partial\Omega_k\setminus \Sigma)$.
We assume that $\Sigma,\Gl,\Gnl$ are Lipschitz submanifolds of $\partial\Ol$ or $\partial\Onl$, respectively. 
We point out that if either $\Gl$ or $\Gnl$ is empty, then $\Sigma$ is a manifold without boundary.

\subsection{Motivation}\label{subsec-Motivation}
As a starting point for our study of the coupled local/nonlocal models, we inspect a classical problem: given $f\in L^2(\Om)$, we consider the purely local energy
\begin{equation} \label{eq:local_energy}
E_\ell[u] := \frac12 \int_{\Ol} \sigma_1  |\nabla u|^2\,dx + \frac12 \int_{\Onl} \sigma_2  |\nabla u|^2\,dx  - \int_\Omega  fu\,dx ,
\end{equation}
where $\sigma$ is a given, bounded, and uniformly positive function and $\sigma_1=\sigma|_{\Ol}$, $\sigma_2=\sigma|_{\Onl}$. Starting from the local energy $E_\ell$, we note that the completion of ${\mathcal D}(\Om)$ with respect to the norm associated with the Euler-Lagrange equations
is equal to $H^1_0(\Om)$ endowed with the norm
\[ 
v \ \mapsto\ \left(\int_\Om \sigma  |\nabla v|^2\,dx\right)^{1/2}.
\]
Thus, the local problem is solved in this functional space. 
The energy $E_\ell$ is convex in $H^1_0(\Omega)$, and its minimizer satisfies the Euler-Lagrange equations: find $u\in H^1_0(\Om)$ such that
\[
\int_\Om \sigma  \nabla u\cdot\nabla v\,dx = \int_\Omega f v \, dx, \qquad \forall v \in H^1_0(\Omega),
\]
or, in strong form,
\[\left\lbrace
\begin{array}{rl}
- \div(\sigma \nabla u) = f & \mbox{in } \Om, \\
u = 0 & \mbox{on } \partial\Omega.
\end{array}
\right. \]
We can also rewrite the previous equation as a transmission problem,
\begin{equation} \label{eq:local-transmission} \left\lbrace
\begin{array}{rl}
- \div(\sigma_1 \nabla u) = f & \mbox{in } \Ol, \\
- \div(\sigma_2 \nabla u) = f & \mbox{in } \Onl , \\
u = 0 & \mbox{on } \partial\Omega, \\
u_{|\Ol} - u_{|\Onl} = 0 & \mbox{on } \Sigma, \\
(\sigma_1 \frac{\partial u}{\partial \vn})_{|\Ol} + (\sigma_2 \frac{\partial u}{\partial \vn})_{|\Onl} = 0 & \mbox{on } \Sigma,
\end{array}
\right. \end{equation}
where $\vn$ denotes the normal derivative to $\Sigma$ pointing outwards from $\Ol$.

A clear fact about the energy $E_\ell$ and its associated problem in $H^1_0(\Omega)$ is that, if the diffusivity jumps at the interface or if it is not sufficiently regular inside $\Ol$ or $\Onl$, then one cannot expect solutions to be globally any smoother than $H^{3/2}(\Omega)$. In contrast, in case of a globally Lipschitz diffusivity and $\partial \Omega$ being of class $C^{1,1}$ or $\Omega$ being convex, one can expect solutions to be in $H^2(\Omega)$. Besides the diffusivities, other aspects to take into account for the sake of the regularity of solutions are the smoothness of the resulting subdomains and how $\Sigma$ intersects $\partial\Omega$.

In some applications, it is convenient to include a nonlocal region to the model. Nonlocality can act as a means of relaxation. For example, peridynamics \cite{Silling} is a nonlocal reformulation of the basic equations of motion in continuum mechanics; the theory remains valid in the presence of discontinuities and is therefore able to cover fracture. Another example consists in the replacement of the usual $H^1$-forms by the corresponding $H^s$-forms ($0<s<1$) in scalar problems. For transmission problems with diffusivities with opposite sign and separated by an interface with a corner, in \cite{BoCi17}, we provided numerical evidence that such an approach leads to a shrinkage of the so-called critical interval for the ratio between the diffusivities, in which the Fredholm solvability is lost.

In this work, we seek an energy-based coupling between a local and a non-local model. Once the energy of the system that includes local and nonlocal contributions is defined, we derive the problem to be solved by using the Euler-Lagrange equations. This variational formulation is posed in the completion of ${\mathcal D}(\Om)$ with respect to the norm associated with these equations. We study the regularity of the solution of the resulting coupled problem and exploit it to prove a priori convergence rates for finite element discretizations of such a problem.

\subsection{Local/nonlocal coupling} \label{sec:l-nl-coupling}
Here, we propose two different energies that extend \eqref{eq:local_energy} to a local/nonlocal coupled setting. These energies replace the weighted $H^1(\Onl)$-seminorm in \eqref{eq:local_energy} by certain weighted $H^s$-seminorms. For the sake of clarity, we assume $f \in L^2(\Omega)$. In the following discussions, more precise conditions are introduced. We emphasize that our goal is to study coupled problems with homogeneous Dirichlet boundary conditions, and we will therefore analyze energies defined over ${\mathcal D}(\Om)$. In all cases the constant $C(n,s)$ is the one defined in \eqref{eq:def-cns} below, and we adopt the notation $\Omega^c := \R^n \setminus \overline \Omega$.

We assume that we are given two measurable functions $\sl, \snl$, to which we will refer as {\em diffusivities}, that satisfy $\sigma_{\max}\ge\sl, \snl \ge\sigma_{\min}>0$ a.e.. In the sequel, the domain of $\sl$ shall be $\Ol$, while the domain of $\snl$ varies from one example to the other but is a subset of $\rn \times \rn$. In order to have a usable variational structure, we require that $\snl (x,y) = \snl(y,x)$ a.e.; nevertheless, we point out this requirement is not fundamental either for well-posedness or for regularity of solutions. 
We shall make additional assumptions on the regularity of the diffusivities when needed in Section \ref{sec:regularity}.

The \emph{first energy} we consider is
\begin{equation} \label{eq:energy-I}
\begin{split}
E_{I}[u] & :=  \frac12 \int_{\Ol} \sl |\nabla u |^2 dx \\
&+ \frac{C(n,s)}{4} \iint_{\Onl \times \Onl} \snl(x,y) \frac{(u(x)-u(y))^2}{|x-y|^{n+2s}} dy dx
 - \int_\Omega f u \, dx,
\end{split} 
\end{equation}
and we take the completion of ${\mathcal D}(\Om)$ with respect to the norm associated with the Euler-Lagrange equations.
Namely, the energy $E_I$ takes into account weighted $H^1(\Ol)$- and $H^s(\Onl)$-seminorms.

\begin{remark}[issues when $s \in (0,1/2{]}$] \label{rem:issues}
While in principle the energy $E_{I}$ is meaningful for $s\in (0,1)$, for $s \in (0,1/2]$ it has two issues:
it does not admit a unique minimizer and it gives rise to two \emph{decoupled problems}. For the sake of simplicity let us assume $\snl \equiv 1$ and $\sl \equiv 1$; in such a case, the energy above is
\[
E_{I}[u] = \frac12 |u|^2_{H^1(\Ol)}+ \frac12 |u|^2_{H^s(\Onl)} - \int_\Omega f u \, dx. 
\]
Regarding the first issue, because for $s \in (0,1/2]$ we have $H^s(\Onl) \equiv H^s_0(\Onl)$, we observe that the energy-minimization problem would not be well-posed because solutions would be defined up to a constant on $\Onl$. Another way to see this is to notice that $u \mapsto \frac12 |u|^2_{H^s(\Onl)}$ does not induce a norm in $H^s_0(\Onl)$.

The second issue is related to the fact that, for $s \in (0,1/2]$, the energy $E_I$ does not enforce any continuity on $\Sigma$. Therefore, when seeking for critical points of $E_I$ one can on the one hand minimize the functional $u \mapsto \frac12 |u|^2_{H^1(\Ol)} - \int_{\Ol} f u \, dx$ over functions that vanish on $\partial \Omega$ and on the other hand seek critical points of the functional  $u \mapsto \frac12 |u|^2_{H^s(\Onl)} - \int_{\Onl} f u \, dx$.
\end{remark}

The \emph{second energy} we inspect is 
\begin{equation} \label{eq:energy-III}
\begin{split}
E_{II}[u] 
& := \frac12 \int_{\Ol} \sl |\nabla u |^2 dx \\
& + \frac{C(n,s)}{4} \iint_{Q_{\Onl}} \snl(x,y) \frac{(u(x)-u(y))^2}{|x-y|^{n+2s}} dy dx - \int_\Omega f u \, dx ,  
\end{split}
\end{equation}
where $Q_{\Onl} = (\rn \times \rn) \setminus (\Onl^c \times \Onl^c)$, and we also take the completion of ${\mathcal D}(\Om)$ with respect to the norm associated with the Euler-Lagrange equations.
With respect to the previous energy, $E_{II}$ now incorporates the interactions between $\Onl$ and $\Onl^c$. This gives rise to meaningful minimization problems for all $s \in (0,1)$, and this energy does not suffer any of the drawbacks mentioned in Remark \ref{rem:issues}. The presence of the integral over $\Ol\times\Onl$ corresponds to a `flux density' between $\Ol$ and $\Onl$ that enforces continuity across the interface and effectively couples the problems even for $s \in (0,1/2]$.

\begin{remark}[Hilbert spaces]
Both energies $E_I$ and $E_{II}$ induce Hilbert space norms in $\calD(\Omega)$. 
\end{remark}

The outline of the paper is as follows. In the next Section, we review two different definitions of fractional Laplacians, the related integration by part formulas and nonlocal Neumann operators. In Section~\ref{sec_EulerLagrange}, we build the Euler-Lagrange equations in strong form for both energies $E_{I}$ and $E_{II}$. In particular, there is a difference in the way the transmission condition acts on the solutions, depending on the chosen energy. We then study the a priori regularity of the solution, focusing on the second energy $E_{II}$. In a last part (Section~\ref{sec:experiments}), we carry out the numerical analysis for both energies, and provide some numerical results that highlight some salient features of our models. 

Along the manuscript, we use the notation $a \lesssim b$ to denote that $a \le C b$ with a constant $C$ that is independent of both $a$ and $b$; when relevant, we specify dependence.

\section{Fractional Laplacians and their associated Neumann operators}

In this section, we write suitable integration by parts formulas for the weighted operators appearing in the energies we proposed in \S\ref{sec:l-nl-coupling}.

\subsection{Weighted fractional Laplacians}\label{subsec-basics}
We shall make use of fractional-order Sobolev spaces and Besov spaces. We employ the notation from \cite{BoNo23} and refer to that work for elementary properties of these spaces. Just in this subsection, we let $s \in (0,1)$ be given, instead of $s \in (1/2,1)$ elsewhere. 

This work is concerned with {\em fractional Laplacians} with order $s$. In the whole space $\rn$, there is a clear way to define such an operator: it is the pseudo-differential operator with symbols $|\xi|^{2s}$. Among several equivalent characterizations one can provide \cite{Kwasnicki2017}, one can set for sufficiently smooth  $u \colon \rn \to \R$,
\begin{equation}\label{eq:def-ifl}
(-\Delta)^s u (x) := C(n,s) \mbox{ p.v.} \int_\rn \frac{u(x)-u(y)}{|x-y|^{n+2s}} dy, \quad x \in \rn, 
\end{equation}
with
\begin{equation} \label{eq:def-cns}
C(n,s) := \frac{2^{2s} s \Gamma(s+\frac{n}2)}{\pi^{n/2} \Gamma(1-s)}.
\end{equation}
The operator \eqref{eq:def-ifl} arises as the infinitesimal generator of a $2s$-stable L\'evy process, and therefore has been widely employed in a number of applications, including the modeling of market fluctuations \cite{ContTankov} and subsurface flow through porous media \cite{BensonWheatcraft}. Our interest in the fractional Laplacian and related operators stems from the analysis of interface problems involving dielectrics and metamaterials \cite{BoCi17}, where the relaxation of the classical $H^1$-setting by introducing $H^s$-forms --and consequently leading to the use of the fractional Laplacian \eqref{eq:def-ifl}-- yields a reduction of the so-called {\em critical interval} and therefore to a more stable behavior of solutions.

There is not a unique way to define fractional Laplacians on arbitrary domains $\domain \subset \rn$. We briefly comment on three non-equivalent definitions. The {\em spectral fractional Laplacian} corresponds to a non-integer power of the Laplace operator with Dirichlet boundary conditions in the spectral sense. A second approach consists in using definition \eqref{eq:def-ifl}; we refer to such an operator as {\em integral fractional Laplacian.} As opposed to the spectral operator, the integral one does not depend on the domain $\domain$; this definition keeps the probabilistic interpretation of the fractional Laplacian defined over $\rn$.
However, one should also observe that \eqref{eq:def-ifl} involves the values of $u$ {\em outside of $\domain$}. 
A third way to define a fractional Laplacian on a domain $\domain$ consists in restricting the integral in \eqref{eq:def-ifl} to the domain $\domain$ itself. Namely, we define the \emph{censored fractional Laplacian} (or \emph{regional fractional Laplacian}) of order $s$ of a function $u \colon \domain \to \R$ by
\begin{equation}\label{eq:def-rfl}
(-\Delta)_\domain^s u (x) := C(n,s) \mbox{ p.v.} \int_\domain \frac{u(x)-u(y)}{|x-y|^{n+2s}} dy, \quad x \in \rn. 
\end{equation}
This operator arises from taking first variation of the functional $u \mapsto \frac12|u|^2_{H^s(\domain)}$.
While in principle this operator is meaningful for $s \in (0,1)$, the case $s>1/2$ is the most interesting: in such a case, the Dirichlet problem for this operator involves boundary values on $\partial \domain$.

This work is concerned with operators related to \eqref{eq:def-ifl} and \eqref{eq:def-rfl} on bounded, Lipschitz domains in $\rn$. These operators shall arise in the Euler-Lagrange equations related to energies \eqref{eq:energy-I} and \eqref{eq:energy-III}, respectively. 
We let $\snl$ be as in the beginning of Section \ref{sec:l-nl-coupling}. By analogy with \eqref{eq:def-ifl}, let us consider the \emph{weighted integral fractional Laplacian} of order $s$, 
\begin{equation}\label{eq:def-wifl}
(-\Delta_\snl)^s u (x) := C(n,s) \mbox{ p.v.}\int_{\rn} \snl(x,y)\frac{u(x)-u(y)}{|x-y|^{n+2s}} dy.
\end{equation}
The symmetry of $\snl$ is required to have a problem with a variational structure; if one starts from an energy with non-symmetric diffusivity, then the associated operator in the Euler-Lagrange equations corresponds to a symmetrization of such a diffusivity. We can also consider a variant of \eqref{eq:def-rfl} with variable diffusivity, the \emph{weighted censored fractional Laplacian} of order $s$,
\begin{equation}\label{eq:def-wrfl}
(-\Delta_\snl)^s_\domain u (x) :=  C(n,s) \int_{\domain} \snl(x,y)\frac{u(x)-u(y)}{|x-y|^{n+2s}} dy, \quad x \in \domain.
\end{equation}

While operators like \eqref{eq:def-wifl} and \eqref{eq:def-wrfl} are well-defined under much more general assumptions on $\snl$, in practice we will require certain additional smoothness on $\snl$ in order to have integration by parts formulas or to exploit regularity estimates for the solutions of the related Dirichlet problems. We revisit this point in Section \ref{sec:regularity} below.

Finally, we note that the pointwise definition of both \eqref{eq:def-wifl} and \eqref{eq:def-wrfl} is meaningful for sufficiently smooth functions. 
On the one hand, we have $(-\Delta_\snl)^s : \calS(\rn) \to \calS(\rn)$, and we can extend the definition to tempered distributions in a standard fashion, namely $(-\Delta_\snl)^s : \calS'(\rn) \to \calS'(\rn)$ by
\[
\langle (-\Delta_\snl)^s T , \phii \rangle := \langle T , (-\Delta_\snl)^s  \phii \rangle, \quad \forall T \in \calS'(\rn), \, \phii \in \calS(\rn).
\]
On the other hand, the operator $(-\Delta_\snl)^s_\domain$ is nonlocal and does not map neither $ \calS(\rn) $ nor $\calD(\domain)$ onto themselves. For the sake of our analysis, we simply point out that if $ \phii \in \calD(\domain)$ then $(-\Delta_\snl)^s_\domain \phii \in L^2(\domain)$; additionally, integrating by parts (see \eqref{eq:ibp-censored-with-sigma} below) we have
\[
( (-\Delta_\snl)^s_\domain u , \phii ) = (u , (-\Delta_\snl)^s_\domain  \phii ), \quad \forall u , \phii \in \calD(\domain).
\]
Therefore we can extend the definition of  $(-\Delta_\snl)^s_\domain \colon L^2(\domain) \to \calD'(\domain)$ by
\begin{equation} \label{eq:def-delta_domain-dist}
\langle (-\Delta_\snl)^s_\domain u , \phii \rangle := (u , (-\Delta_\snl)^s_\domain  \phii ), \quad \forall u \in L^2(\domain), \, \phii \in \calD(\domain).
\end{equation}

\subsection{Integration by parts formulas for fractional Laplacians}
We next review some integration by parts formulas for the (weighted) integral and censored fractional Laplacians on bounded domains. For the integral operator \eqref{eq:def-ifl}, \cite[Lemma 3.3]{DiRoVa17} derives an identity of this type. A straightforward adaptation of the proof to the variable diffusivity setting gives rise to the following.

\begin{lemma}\label{lemma:ibp-snl}
Let $\domain \subset \R^n$ be bounded and Lipschitz, $\snl\in L^\infty(Q_\domain)$ be symmetric, and bounded functions $u,v \in C_c^2(\R^n)$,
\begin{equation} \label{eq:wibp}
\begin{split}
C(n,s) & \iint_{Q_\domain} \snl(x,y) \frac{(u(x)-u(y))(v(x)-v(y))}{|x-y|^{n+2s}} dy dx   \\
= &  \int_\domain v (-\Delta_\snl)^s u\,dx + \int_{\domain^c} v\,  \Nsnl^\domain u \,dx,
\end{split}
\end{equation}
where 
\begin{equation} \label{eq:def-wNs}
 \Nsnl^\domain u (x) := C(n,s) \int_\domain \snl(x,y) \frac{u(x)-u(y)}{|x-y|^{n+2s}} dy, \quad x \in \domain^c
\end{equation}
is a weighted nonlocal Neumann operator.
\end{lemma}

The operator \eqref{eq:def-wNs}, with $\snl \equiv 1$, has been proposed in \cite{DiRoVa17}; while there is no widely accepted definition of a Neumann condition for the integral fractional Laplacian \eqref{eq:def-ifl}, the one we chose has the advantage of giving rise to a variational structure while keeping the probabilistic interpretation of a Neumann condition as a random reflection --according to a L\'evy flight--  of a particle inside the domain. We refer to \cite[Section 7]{DiRoVa17} for a thorough comparison between this and other existing approaches.

As for the censored fractional Laplacian, according to \cite[Theorem 3.3]{Guan} and \cite[Theorem 2.7]{GalWarma} (see also \cite[Theorem 5.7]{Warma} for a more general version), we have the following integration by parts formula. Given a bounded $C^{2}$ domain $\domain$, if $s \in (1/2,1)$, $v \in H^s(\domain)$, and $u \in C(\overline\domain)$ can be written as $u (x) = \phi(x) d(x,\partial\domain)^{2s-1} + \psi(x)$ for all $x\in \domain$, with $\phi,\psi \in C^2(\overline\domain)$, then 
\begin{equation} \label{eq:ibp-censored}
\begin{split}
 \frac{C(n,s)}{2} & \iint_{\domain \times \domain} \frac{(u(x)-u(y))(v(x)-v(y))}{|x-y|^{n+2s}} dy dx   \\
= & \int_\domain v (-\Delta)^s_\domain u \,dx + \int_{\partial\domain} v \Ms^\domain u \,d\sigma.
\end{split}\end{equation}
Above, $\Ms^\domain$ is the fractional normal derivative 
\begin{equation} \label{eq:def-Ms}
\Ms^\domain u (x) :=  C'(s) \lim_{t\to 0^+} \frac{u(x-t\vn)-u(x)}{t^{2s-1}}, \quad x \in \partial\domain,
\end{equation}
where $\vn$ is the outward normal on $\partial\domain$, and $C'(s)$ is the constant introduced in \cite[(5.13)-(5.14)]{Warma} or \cite[(3.8)-(3.9)]{Guan}, 
\[
C'(s) := \frac{C(n,s)}{2s(2s-1)} \left(\int_0^\infty \frac{|t-1|^{1-2s} - \max\{1,t\}^{1-2s}}{t^{2-2s}} \, dt \right) \left(\int_{ \{|x|=1, x_n > 0\}} x_n^{2s} \, dx \right).
\]

We are interested in the extension of \eqref{eq:ibp-censored} to the variable diffusivity case. Reference \cite[Theorem 4.5]{Guan} provides sufficient conditions for such an integration by parts formula to hold. If $\snl \in C^1(\overline\domain)$, then taking $\psi_1 = \snl$, $\psi_2 = 0$ in that theorem allows one to get the formula below with
\[
\Msnl^{\domain} u (x) := \snl(x,x) \Ms^{\domain} u (x),
\]
namely
\begin{equation} \label{eq:ibp-censored-with-sigma}
\begin{split}
\frac{C(n,s)}{2} & \iint_{\domain\times\domain} \snl(x,y) \frac{(u(x)-u(y))(v(x)-v(y))}{|x-y|^{n+2s}} dy dx  \\ 
= & \int_{\domain} v (-\Delta_\snl)^s_{\domain} u \,dx+ \int_{\partial{\domain}} v \Msnl^{\domain} u \,d\sigma. 
\end{split}
\end{equation}

\subsection{Traces and liftings (or restrictions and extensions)} \label{sec:traces-liftings}
In view of the integration by parts formula \eqref{eq:wibp}, to realize the mapping properties of the weighted Neumann operator $\Nsnl^\domain u$ it suffices to characterize the trace space associated to the energy. Such an effort has been carried out, for instance, in \cite{Bogdan,GrHe22} by using robust Poisson kernel estimates. Here, we follow reference \cite{GrKa23} and introduce the space
\begin{equation}\label{eq:def-trace-space}
\calT^s(\domain^c) := \{ g \colon \domain^c \to \R \mbox{ measurable } \colon \| g \|_{\calT^s(\domain^c)} < \infty \},
\end{equation}
where
\[ \begin{aligned}
\| g \|_{\calT^s(\domain^c)} & := \left( \| g \|_{L^2(\domain^c; \mu_s)}^2 + | g |_{\calT^s(\domain^c)}^2\right)^{1/2} , \\
| g |_{\calT^s(\domain^c)} & := \left(\iint_{\domain^c \times \domain^c} \frac{(g(x) - g(y))^2}{((|x-y| + d_x + d_y ) \wedge 1)^n} \, \mu_s(dx) \mu_s(dy)\right)^{1/2},
\end{aligned} \]
$d_x = \mbox{dist}(x, \partial \domain)$, and
\[
\mu_s(dx) := \frac{(1-s) \chi_{\domain^c}(x)}{d_x^s (1+d_x)^{n+s}} \,  dx.
\]
To deal with the natural space for the  energy $E_{II}$, it shall be convenient to introduce the space
\[
V^s(\domain | \R^n) := \{ u \colon \R^n \to \R \mbox{ measurable } \colon \| u \|_{V^s(\domain | \R^n)} < \infty \},
\]
endowed with 
\begin{equation} \label{eq:def-Vs-norm} \begin{aligned}
\| u \|_{V^s(\domain | \R^n)} & := (\| u \|_{L^2(\domain)}^2 + |u|_{V^s(\domain | \R^n)}^2)^{1/2}, \\
 |u|_{V^s(\domain | \R^n)}& := \left(\frac{C(n,s)}{2} \iint_{Q_\domain}\frac{(u(x)-u(y))^2}{|x-y|^{n+2s}} dy dx\right)^{1/2}.
\end{aligned} \end{equation}

\begin{remark}[boundedness of $V^s$-inner product]\label{rem:Cauchy-Schwarz}
Let $t \in [s, \min\{1,2s\})$. Then, a straightforward application of the Cauchy-Schwarz inequality yields
\[
\frac{C(n,s)}{2} \iint_{Q_\domain}  \snl(x,y) \frac{(u(x)-u(y))(v(x)-v(y))}{|x-y|^{n+2s}} dy dx \lesssim |u|_{V^t(\domain | \R^n)} |v|_{V^{2s-t}(\domain | \R^n)}
\]
for all $u \in V^{t}(\domain | \R^n), \, v \in V^{2s-t}(\domain | \R^n)$, with a constant depending on $n,s,t, \snl$.
\end{remark}

With the notation above, we have the following trace (restriction) and lifting (extension) result \cite[Theorem 1.2]{GrKa23}.

\begin{theorem} \label{thm:trace-extension}
Let $\domain$ be a bounded Lipschitz domain and $t \in (0,1)$.
\begin{enumerate}
\item There exists a continuous trace operator ${\rm Tr}_t \colon V^t(\domain | \R^n) \to \calT^t(\domain^c)$.

\item There exists a continuous lifting operator ${\rm Ext}_t \colon \calT^t(\domain^c) \to V^t(\domain | \R^n)$.
\end{enumerate}
In both cases, the continuity constants depend on $\Omega$ and on a lower bound on $t$, namely, they are uniformly bounded as $t \to 1$.
\end{theorem}

\subsection{On the domain of the nonlocal Neumann operator.} \label{sec:domain-Nsnl}
We can combine Theorem \ref{thm:trace-extension} with the integration by parts formula \eqref{eq:wibp} to extend the operator $\Nsnl^\domain$ to a broader class of functions than $C^2_c(\R^n)$.
Indeed, let $t \in [s, \min\{1,2s\})$ and $\tau \in [0, \min\{1/2, 2s - t\} )$. We define the space
$\mathcal{H}^t(\domain)$ through 
\begin{equation}\label{eq:def-calHt}
\mathcal{H}^t(\domain) = \overline{C^\infty_c(\R^n)}^{\| \cdot \|_{\mathcal{H}^t(\domain)}}, \quad
\| u \|_{\mathcal{H}^t(\domain)} := \| u \|_{V^t(\domain | \R^n)} + \| (-\Delta_\snl)^s u \|_{H^{-\tau}(\domain)}.
\end{equation}
Here, because $\tau<1/2$, we understand the restriction of the distribution $(-\Delta_\snl)^s u$ in the dual space of $\widetilde{H}^\tau(\domain)$, namely, acting on functions that are extended by zero on $\domain^c$. 
We let $L \colon \mathcal{H}^t(\domain) \times V^{2s-t}(\domain | \R^n) \to \R$,
\begin{equation} \label{eq:def-L} \begin{split} 
L(u,v) := & \frac{C(n,s)}{2} \iint_{Q_\domain}  \snl(x,y) \frac{(u(x)-u(y))(v(x)-v(y))}{|x-y|^{n+2s}} dy dx \\ & - \int_\domain v (-\Delta_\snl)^s u\,dx .
\end{split} 
\end{equation}
By Remark \ref{rem:Cauchy-Schwarz}, $L$ is a bounded, linear operator:
\begin{equation} \label{eq:L-bounded} \begin{split}
|L(u,v)| & \lesssim \left( |u|_{V^t(\domain | \R^n)} |v|_{V^{2s-t}(\domain | \R^n)} + \| v \|_{H^\tau(\domain)}\| (-\Delta_\snl)^s u \|_{H^{-\tau}(\domain)} \right) \\
& \lesssim \|u \|_{\mathcal{H}^t(\domain)} \|v\|_{V^{2s-t}(\domain | \R^n)},
\end{split}\end{equation}
because $\| v \|_{H^\tau(\domain)} = (\| v \|_{L^2(\domain)}^2 + | v |_{H^\tau(\domain)}^2)^{1/2} \lesssim \|v\|_{V^{2s-t}(\domain | \R^n)}$ since $\tau \le 2s-t$ and $\domain$ is a bounded, Lipschitz set.

Next, if $u \in C^2_c(\R^n)$ and $v \in \mathcal{D}(\domain)$, we can use the integration by parts formula from Lemma \ref{lemma:ibp-snl} to observe
\begin{equation} \label{eq:identity-L-vanishes}
L(u,v) =\int_{\domain^c} v \Ns^\domain u\,dx = 0
\end{equation}
and, by density, we deduce that $L(u,v) = 0$ for all $u \in \mathcal{H}^t(\domain)$, $v \in \widetilde{H}^{2s-t}(\domain)$. 

We define $L_{\domain^c} \colon \mathcal{H}^t(\domain) \times \calT^{2s-t}(\domain^c)$,
\[
L_{\domain^c}(u,g) = L(u, {\rm Ext}_{2s-t}(g) ),
\]
where ${\rm Ext}_{2s-t}  \colon \calT^{2s-t}(\domain^c) \to  V^{2s-t}(\domain | \R^n)$ is an extension operator as in Theorem \ref{thm:trace-extension}. It is clear that $L_{\domain^c}$ does not depend on how we extend $g$ to $\domain$: any other extension ${\rm Ext'}(g) \in V^{2s-t}(\domain | \R^n)$ would yield ${\rm Ext}_{2s-t}(g) - {\rm Ext'}(g) \in \widetilde{H}^{2s-t}(\domain)$ and therefore $L(u,  {\rm Ext'}(g)) = L(u,  {\rm Ext}_{2s-t}(g))$.

Given $u \in \calH^t(\domain)$, $g \in \calT^{2s-t}(\domain^c)$, we then have by \eqref{eq:L-bounded},
\[
|L_{\domain^c}(u,g)| \lesssim \| u \|_{\mathcal{H}^t(\domain)}  \|  {\rm Ext}_{2s-t}(g) \|_{V^{2s-t}(\domain | \R^n)} \lesssim \| u \|_{\mathcal{H}^t(\domain) }  \| g \|_{ \calT^{2s-t}(\domain^c)}.
\]
The preceding discussion leads us to defining, for any $t \in [s, \min \{1, 2s\})$,
\[
\Nsnl^\domain \colon \mathcal{H}^t(\domain) \to \left( \calT^{2s-t}(\domain^c)\right)'
\]
by
\begin{equation} \label{eq:def-Nsnl}
\langle \Nsnl^\domain u, g \rangle := L_{\domain^c}(u,g).
\end{equation}
Naturally, by construction this is a bounded operator; additionally, if $u$ is sufficiently smooth we can identify $\Nsnl^\domain$ with its pointwise definition, and the integration by parts formula remains valid.

\begin{remark}[role of $t$] We introduced a parameter $t \in  [s, \min \{1, 2s\})$ in the previous definition of the nonlocal Neumann operator $\Nsnl^\domain$. Unlike Sobolev spaces defined on bounded domains, the spaces $V^s(\domain|\R^n)$ are not ordered with respect to $s$: in general, it is not true that $V^t(\domain|\R^n) \subset V^s(\domain|\R^n)$ if $t > s$ because the former space can accommodate functions with a stronger growth at infinity than the latter. Therefore, formally, the operator \eqref{eq:def-Nsnl} depends on $t$. However, for the sake of this work, if we restrict our attention to functions that have bounded support, then $t$ plays no significant role as all the $t$-dependent operators coincide with the pointwise definition \eqref{eq:def-wNs} for functions in $C^\infty_c(\R^n)$.
\end{remark}

\begin{remark}[role of $\tau$]
To define the nonlocal Neumann operator $\Nsnl^\domain u$, we not only required the function $u$ to belong to the natural energy space $V^s(\domain | \R^n)$ but also that  $(-\Delta_\snl)^s u  \in H^{-\tau}(\domain)$ for some $\tau < 1/2$. This is in alignment with the classical (local) case, in which one cannot define normal derivatives of functions on $H^1(\Omega)$ but requires some integrability of their Laplacian (see, for example, \cite[Chap. 6]{Sayas-book}).
\end{remark}
\section{Euler-Lagrange equations}\label{sec_EulerLagrange}
For the two energies we introduced in \S\ref{sec:l-nl-coupling}, that extend \eqref{eq:local_energy} to a coupled local/nonlocal setting, we discuss the strong form of the resulting Euler-Lagrange equations. We recall that $\sl, \snl$ are bounded and uniformly positive functions.

Conceptually, we follow a standard procedure: we first consider the first variation of the energies with respect to smooth functions supported in either $\Ol$, $\Onl$, to obtain suitable integro-differential PDEs on such domains, and afterwards with respect to smooth functions whose support intersects $\Sigma$ to derive transmission conditions. In doing so, we find some remarkable differences between these energies.

\subsection{Euler-Lagrange equation for $E_{I}$}\label{subsec_Euler-Lagrange_I}
We begin with the energy $E_{I}$ in \eqref{eq:energy-I}. We assume $\Onl$ to be a $C^2$ domain, so that the integration by parts formula \eqref{eq:ibp-censored-with-sigma} holds. The Euler-Lagrange equations associated to $E_I$ write: find $u\in \calH_{I}$ such that
\begin{equation} \label{eq:Euler-Lagrange-E_I} \begin{aligned}
 \int_{\Ol} \sl \nabla u \cdot \nabla v\, dx + \frac{ C(n,s) }{2} \iint_{\Onl^2} \snl(x,y) \frac{(u(x)-u(y))(v(x)-v(y))}{|x-y|^{n+2s}} dy dx  \\
= \int_\Omega f v \, dx, \qquad \forall v \in \mathcal{H}_{I},
\end{aligned}
\end{equation}
where $\mathcal{H}_{I}$ is the completion of ${\mathcal D}(\Om)$ with respect to the norm associated with the energy,
\[ v \ \mapsto\ \int_{\Ol} \sl |\nabla v|^2\, dx+ \left( \frac{ C(n,s) }{2} \iint_{\Onl\times\Onl} \snl(x,y) \frac{(v(x)-v(y))^2}{|x-y|^{n+2s}} dy dx \right)^{1/2}.\]

\smallskip

\noindent
\underline{Step 1}. 
We take $v\in{\mathcal D}(\Ol)$ in \eqref{eq:Euler-Lagrange-E_I} and integrate by parts to find that
\[ \langle -\div (\sl \nabla u), v \rangle = \int_{\Ol} \sl \nabla u \cdot \nabla v\, dx = \int_{\Ol} f v \, dx \qquad \forall v \in {\mathcal D}(\Ol). \]
Thus, in ${\mathcal D}'(\Ol)$, and even in $L^2(\Ol)$ since $f_{|\Ol}\in L^2(\Ol)$,
\begin{equation} \label{eq:u-Hdiv}
-\div (\sl \nabla u) = f \mbox{ in }\Ol. 
\end{equation}

\smallskip

\noindent
\underline{Step 2}. 
Next, we take $v\in{\mathcal D}(\Onl)$ in \eqref{eq:Euler-Lagrange-E_I} to obtain, for all $v \in {\mathcal D}(\Onl)$,
\[  \frac{C(n,s)}{2} \iint_{\Onl\times\Onl} \snl(x,y) \frac{(u(x)-u(y))(v(x)-v(y))}{|x-y|^{n+2s}} dy dx = \int_{\Onl} f v \, dx. \]
On the other hand, using \eqref{eq:def-delta_domain-dist} and the fact that $\mathcal{H}_{I} \subset H^s(\Onl)$ so that we can apply \eqref{eq:ibp-censored-with-sigma}, we have
\[ \frac{ C(n,s) }{2} \iint_{\Onl\times\Onl} \snl(x,y) \frac{(u(x)-u(y))(v(x)-v(y))}{|x-y|^{n+2s}} dy dx = \langle  (-\Delta_\snl)^s_{\Onl} u, v \rangle.\]
Thus, in ${\mathcal D}'(\Onl)$, and even in $L^2(\Onl)$ since $f_{|\Onl}\in L^2(\Onl)$,
\[ 
(-\Delta_\snl)^s_{\Onl} u = f \mbox{ in }\Onl. 
\]

\smallskip

\noindent
\underline{Step 3}. 
Finally, let us take $v\in{\mathcal D}(\Omega)$ in \eqref{eq:Euler-Lagrange-E_I}. Formula \eqref{eq:u-Hdiv} yields $\sl \nabla u |_{\Ol} \in \mathbf{H}(\div,\Ol)$, and we therefore have
\[ \begin{split}
\int_{\Ol} \sl \nabla u \cdot \nabla v\, dx & = - \int_{\Ol} \div(\sl \nabla u ) v\, dx + \langle (\sl \frac{\partial u}{\partial \vn})_{|\Ol}, v \rangle_{\widetilde{H}^{1/2}(\Sigma)} \\
& = \int_{\Ol} f  v\, dx + \langle (\sl \frac{\partial u}{\partial \vn})_{|\Ol}, v \rangle_{\widetilde{H}^{1/2}(\Sigma)},
\end{split} \]
where $\widetilde{H}^{1/2}(\Sigma):= \{ g \in H^{1/2}(\partial \Ol) \colon g |_{\Gl} = 0 \} = \{ g \in H^{1/2}(\partial \Onl) \colon g |_{\Gnl} = 0 \}$, and the equality holds algebraically and topologically. Since $(-\Delta_\snl)^s_{\Onl} u\in L^2(\Om_2)$, we can apply the integration by parts formula \eqref{eq:ibp-censored-with-sigma} and replace in \eqref{eq:Euler-Lagrange-E_I} to arrive at
\[ \langle (\sl \frac{\partial u}{\partial \vn})_{|\Ol}, v \rangle_{\widetilde{H}^{1/2}(\Sigma)} + \int_{\Sigma} v \Msnl^{\Om_2} u \,d\sigma = 0, \quad \forall v \in {\mathcal D}(\Omega) .\]
In other words, we have shown that
\[ (\sl \frac{\partial u}{\partial \vn})_{|\Ol} + \Msnl^{\Om_2} u = 0 \mbox{ on } \Sigma. \]

To conclude, the minimization problem associated with the energy $E_I$ corresponds to the following:
find $u\in \mathcal{H}_{I}$ such that
\begin{equation}\label{eq:coupled_problemI_NEW} \tag{CP-I-NEW}
 \left\lbrace \begin{aligned}
-\div (\sl \nabla u) = f  & \mbox{ in } \Ol, \\
(-\Delta_\snl)^s_{\Onl} u = f  & \mbox{ in } \Onl, \\
(\sl \frac{\partial u}{\partial \vn})_{|\Ol} + (\Msnl^{\Om_2} u)_{|\Onl} = 0 & \mbox{ on } \Sigma.
\end{aligned} \right.
\end{equation}
Thus, for this energy the coupling between the local and the nonlocal models occurs through the interface $\Sigma$. Remarkably, the resulting transmission condition involves the fractional normal derivative operator $\Msnl^{\Om_2}$.

\subsection{Euler-Lagrange equation for $E_{II}$}\label{subsec_Euler-Lagrange_III}
For the energy $E_{II}$ in \eqref{eq:energy-III}, the Euler-Lagrange equations write: find $u\in \mathcal{H}_{II}$ such that
\begin{equation} \label{eq:weak-EIIv3} \begin{aligned}
\int_{\Ol}  \sl \nabla u \cdot \nabla v\, dx + \frac{C(n,s)}{2}  \iint_{Q_{\Onl}} \snl(x,y) \frac{(u(x)-u(y))(v(x)-v(y))}{|x-y|^{n+2s}} dy dx\\ 
= \int_\Omega f v \, dx, \quad \forall v \in \mathcal{H}_{II},
\end{aligned} \end{equation}
where $\mathcal{H}_{II}$ is the completion of ${\mathcal D}(\Om)$ with respect to the norm associated with the Euler-Lagrange equations
\begin{equation} \label{eq:norm-EII}
v \ \mapsto\ \left(  \int_{\Ol} \sl |\nabla v|^2\, dx + \frac{C(n,s)}{2} \iint_{Q_{\Onl}} \snl(x,y) \frac{(v(x)-v(y))^2}{|x-y|^{n+2s}} dy dx\right)^{1/2}
\end{equation}
which, since $\sigma_{\max}\ge\sl, \snl \ge\sigma_{\min}>0$ a.e. and using the notation \eqref{eq:def-Vs-norm}, coincides with
\[
\left( | v |_{H^1(\Ol)}^2 + | v |_{V^s(\Onl | \R^n)}^2 \right)^{1/2} .
\]
We repeat the same three steps as in \S\ref{subsec_Euler-Lagrange_I}.
\smallskip

\noindent
\underline{Step 1}. Taking $v\in{\mathcal D}(\Ol)$, the integral over $\Ol$ in \eqref{eq:weak-EIIv3} is handled as before by a standard integration by parts. On the other hand, there is a contribution coming from the double integral over $Q_{\Onl}$, namely over $(\Ol\times\Onl)\cup(\Onl\times\Ol)$. Recalling the definition \eqref{eq:def-wNs} of the weighted  nonlocal Neumann operator, which acts as a flux density from $\Onl$ to $\Ol$, we conclude that, in $L^2(\Ol)$,
\[ -\div (\sl \nabla u) + \Nsnl^{\Onl} u = f \mbox{ in }\Ol. \]

\smallskip

\noindent
\underline{Step 2}. Taking $v\in{\mathcal D}(\Onl)$ in \eqref{eq:def-L} and using \eqref{eq:identity-L-vanishes}, one has
\[ \frac{C(n,s)}{2}\iint_{Q_{\Onl}} \snl(x,y)  \frac{(u(x)-u(y))(v(x)-v(y))}{|x-y|^{n+2s}} dy dx  = \int_{\Onl} f v \, dx.
\]
On the other hand, since $v$ vanishes on $\Onl^c$, formula \eqref{eq:def-Nsnl} yields
\[ \frac{C(n,s)}{2} \iint_{Q_{\Onl}} \snl(x,y) \frac{(u(x)-u(y))(v(x)-v(y))}{|x-y|^{n+2s}} dy dx = \langle  (-\Delta_\snl)^s u, v \rangle. \]
Thus, we obtain that in ${\mathcal D}'(\Onl)$, and even in $L^2(\Onl)$ since $f_{|\Onl}\in L^2(\Onl)$,
\[ (-\Delta_\snl)^s u = f \mbox{ in }\Onl. \]

\subsubsection*{A Neumann-type operator} To obtain the Euler-Lagrange equations for $E_{II}$, we still need to consider $v\in{\mathcal D}(\Omega)$ in \eqref{eq:weak-EIIv3} and perform the corresponding integration by parts. However, over $\Ol$, we only know that the sum $\calA u := -\div (\sl \nabla u) + \Nsnl^{\Onl} u$ belongs to $L^2(\Ol)$ but not the terms separately.
We shall proceed similarly to the second part of \S\ref{sec:traces-liftings} to derive a integration by parts formula for $\calA$ and deduce the form of its associated Neumann operator. Indeed, we let $\mathcal{J}^s(\Onl; \Ol)$ be the closure of $C^\infty_c(\R^n)$ with respect to
the norm
\[
\| u \|_{\mathcal{J}^s(\Onl;\Ol)} := \left( \| u \|_{V^s(\Onl | \R^n)}^2 +  \| (-\Delta_\snl)^s u \|_{L^2(\Onl)}^2 + \| u \|_{H^1(\Ol)}^2 +  \| \calA u \|_{L^2(\Ol)}^2 \right)^{\frac12}.
\]
We remark that, above, similarly to \S\ref{sec:domain-Nsnl}, instead of $\| (-\Delta_\snl)^s u \|_{L^2(\Onl)}$, for the sake of the following argument we could have equivalently used $\| (-\Delta_\snl)^s u \|_{H^{-\tau}(\Onl)}$ for any $\tau \in (0,1/2)$. We point out that $\mathcal{J}^s(\Onl;\Ol) \subset \mathcal{H}^s(\Onl)$ (see definition \eqref{eq:def-calHt}), and therefore for all $u \in  \mathcal{J}^s(\Onl;\Ol) $ we have
 $\Nsnl^{\Onl} u \in \left( \calT^{s}(\Onl^c)\right)'$ by \eqref{eq:def-Nsnl}.

We let $H^1_{0,\Gl}(\Ol)$ be the completion of
\[ \{ v \in C^\infty_c(\overline{\Ol}) \colon v=0 \mbox{ in a neighborhood of } \Gl\}  \]
in $H^1(\Ol)$. Since $\Gl$ is a Lipschitz submanifold of $\partial\Ol$, we know that $H^1_{0,\Gl}(\Ol)$ is algebraically and topologically equal to the set of functions in $H^1(\Ol)$ that have zero trace on $\Gl$. Then, we define a bilinear operator $L \colon \mathcal{J}^s(\Onl; \Ol) \times H^1_{0,\Gl}(\Ol) \to \R$,
\begin{equation} \label{eq:def-L-II}
L(u,v) := \langle \Nsnl^{\Onl} u , v \rangle
- \int_{\Ol} \mathcal{A} u v + \int_{\Ol} \sl \nabla u \cdot \nabla v .
\end{equation}
We claim that $L$ is well-defined. Given $u \in \mathcal{J}^s(\Onl;\Ol) $ and $v \in H^1_{0,\Gl}(\Ol)$, by the definitions of these spaces the only term in the formula above that we need to inspect with detail is the first one and, since $\Nsnl^{\Onl} u \in \left( \calT^{s}(\Onl^c)\right)'$, it suffices to show that $v \in \calT^{s}(\Onl^c)$. Because  $v \in H^1_{0,\Gl}(\Ol)$, we can extend it first to $H^1_0(\Omega)$ by lifting  $v|_{\Sigma}\in\widetilde{H}^{1/2}(\Sigma)$ to $\Onl$ and then by zero on $\Omega^c$. Denoting by $\tilde E v$ be such an extension, we  find that $\tilde Ev \in \calT^{s}(\Onl^c)$ because $v|_{\Onl^c} \in H^s(\Onl^c)$, and $\| \tilde E v  \|_{\calT^{s}(\Onl^c)} \lesssim \|v \|_{H^1(\Ol)}$. We remark that, according to \eqref{eq:def-Nsnl}, the way $v$ is extended to $\Onl$ is irrelevant for the computation of the duality term $\langle \Nsnl^{\Onl} u , v \rangle$.

From the preceding discussion, it follows that the operator $L$ is bounded:
\[ \begin{split}
| L(u,v) | & \le \| \Nsnl^{\Onl} u  \|_{\left(\calT^{s}(\Onl^c)\right)'}  \| \tilde E v  \|_{\calT^{s}(\Onl^c)} \\ & \quad +  \| \mathcal{A} u \|_{L^2(\Ol)} \| v \|_{L^2(\Ol)} \\ & \quad  + \| \sl \|_{L^\infty(\Ol)} |u|_{H^1(\Ol)} |v|_{H^1(\Ol)} \\ 
 & \lesssim \| u \|_{\mathcal{J}^s(\Onl;\Ol)} \|v \|_{H^1(\Ol)},
\end{split}\]
where we used the fact that 
\[
\| \Nsnl^{\Onl} u  \|_{\left(\calT^{s}(\Onl^c)\right)'}\lesssim \left( \| u \|_{V^s(\Onl | \R^n)} +  \| (-\Delta_\snl)^s u \|_{L^2(\Onl)} \right).
\]

Similarly to \S\ref{sec:domain-Nsnl}, if $u \in C^2(\overline\Omega)$ and $v \in \mathcal{D}(\Ol)$ we have, for the operator $L$ in \eqref{eq:def-L-II}, that $L(u,v) = 0$. Indeed, we can split the operator $\calA u = -\div (\sl \nabla u) + \Nsnl^{\Onl} u$, cancel out the integrals involving  $\Nsnl^{\Onl} u$, and integrate by parts the local term. Therefore, $L$ vanishes on $\mathcal{J}^s(\Onl;\Ol) \times H^1_0(\Ol)$.

Next, we use the operator $L$  to define a {\em local Neumann operator} over $\Sigma$ for functions in $\mathcal{J}^s(\Onl;\Ol)$, which we shall denote by $D_\Sigma$ and which we expect to be in some space with differentiability order $-1/2$\footnote{This is by analogy with the standard Neumann operator, that maps functions on $H^1(\Omega)$ whose Laplacian belongs to $L^2(\Omega)$ to $(\widetilde{H}^{1/2}(\Sigma))'$, cf. \cite[Chap. 6]{Sayas-book}, for example.}. Given $g \in \widetilde{H}^{1/2}(\Sigma)$, we lift it to $Eg \in H^1_{0,\Gl}(\Ol)$ and we can then define $D_\Sigma \colon\mathcal{J}^s(\Onl;\Ol) \to (\widetilde{H}^{1/2}(\Sigma))'$ by 
\[
\langle D_\Sigma u , g \rangle := L(u,Eg).
\]
It is clear that this definition does not depend on the lifting operator $E$:  given any two liftings of $g \in \widetilde{H}^{1/2}(\Sigma)$, $E_1 g, E_2 g \in H^1_{0,\Gl}(\Ol)$, we have $E_1 g - E_2 g \in H^1_0(\Ol)$ and thus $L(u,E_1g - E_2 g) = 0$ for all $u \in \mathcal{J}^s(\Onl;\Ol)$.

\smallskip

We are now in position to proceed with the derivation of the Euler-Lagrange equations associated to $E_{II}$. 

\smallskip

\noindent
\underline{Step 3.} From Step 2, we have $(-\Delta_\snl)^s u \in L^2(\Ol)$ and therefore $u \in \mathcal{J}^s(\Onl;\Ol)$. Thus,
we take $v \in \mathcal{D}(\Omega)$ in \eqref{eq:weak-EIIv3}, use the integration by parts formula from Lemma \ref{lemma:ibp-snl} 
and the definition \eqref{eq:def-L-II} of the operator $L$:
\[ \begin{split}
& \int_\Omega f v \, dx \\
& = \frac{C(n,s)}{2} \iint_{Q_{\Onl}} \snl(x,y) \frac{(u(x)-u(y))(v(x)-v(y))}{|x-y|^{n+2s}} dy dx + \int_{\Ol}  \sl \nabla u \cdot \nabla v\, dx \\
& = \int_{\Onl} v (-\Delta_\snl)^s u + \int_{\Ol} v \mathcal{A} u + L(u, v) .
\end{split} \]
Because $(-\Delta_\snl)^s u = f$ in $\Onl$ and $\mathcal{A} u = f$ in $\Ol$, we get $L(u, v) = 0$, and we recover the interface condition
\[
\langle D_\Sigma u , v \rangle_{ \widetilde{H}^{1/2}(\Sigma)}  = 0 \quad \forall v \in \mathcal{D}(\Omega).
\]

To conclude, the solution $u$ to our problem with the second energy is governed by:
find $u\in \mathcal{H}_{II}$ such that
\begin{equation}\label{eq:coupled_problemIII_NEW} \tag{CP-II-NEW}
 \left\lbrace \begin{aligned}
-\div (\sl \nabla u) + \Nsnl^{\Onl} u = f  & \mbox{ in } \Ol, \\
(-\Delta_\snl)^s u = f  & \mbox{ in } \Onl, \\
D_\Sigma u = 0 & \mbox{ on } \Sigma.
\end{aligned} \right.
\end{equation}

Let $u \in \mathcal{J}^s(\Onl;\Ol)$. Since $\calA u \in L^2(\Ol)$, if either $\div (\sl \nabla u)$ or  $\Nsnl^{\Onl} u$ belong to $L^2(\Ol)$, then the other one belongs to $L^2(\Ol)$ as well. In such a case,  $D_\Sigma u$ coincides with the (outer) normal derivative of $u$ on $\Sigma$ from $\Ol$. Indeed, splitting $\calA u$, canceling out the integrals involving $\Nsnl^{\Onl} u$, and integrating by parts the local term we reach
\[
\langle D_\Sigma u , g \rangle_{ \widetilde{H}^{1/2}(\Sigma)} = L(u,Eg) =  \langle (\sl \frac{\partial u}{\partial \vn})_{|\Ol}, E g \rangle_{H^{1/2}(\partial\Ol)} =  \langle (\sl \frac{\partial u}{\partial \vn})_{|\Ol}, g \rangle_{ \widetilde{H}^{1/2}(\Sigma)}.
\]
We therefore have for all $u \in \mathcal{J}^s(\Onl;\Ol)$
\begin{equation} \label{eq:normal-derivative}
D_\Sigma u = (\sl \frac{\partial u}{\partial \vn})_{|\Ol}\mbox{ provided that } 
\div (\sl \nabla u)\in L^2(\Ol) \mbox{ or } \Nsnl^{\Onl} u\in L^2(\Ol).
\end{equation} 

\begin{remark}[Neumann boundary condition] 
\label{rem:homogeneous-Neumann}
If $s< 1/2$ and $f|_{\Onl}$ is smoother than just $L^2$, then we actually expect the solutions to exhibit a standard homogeneous Neumann condition on the interface. 
Indeed, in that case one can show that $ \Nsnl^{\Onl} u |_{\Ol} \in L^2(\Ol)$.  
This, in turn, allows us to apply known results for local operators with mixed boundary conditions to prove the regularity of solutions.
We postpone a full discussion of this topic to Section \ref{sec:regularity-small-s}.
\end{remark}

In general, we must point out that we expect neither $\div (\sl \nabla u)$ nor $\Nsnl^{\Onl} u$ to belong to $L^2(\Ol)$ and, therefore, we do not expect the boundary condition $D_\Sigma u = 0$ on $\Sigma$ to correspond to a classical homogeneous Neumann condition.
This is consistent with the fact that, as $s \to 1$, the system \eqref{eq:coupled_problemIII_NEW} should formally converge to \eqref{eq:local-transmission}.

\section{Regularity of solutions for the energy $E_{II}$.} \label{sec:regularity}
In this section, we shall focus on the energy $E_{II}$. We address the regularity of the minimizers to \eqref{eq:energy-III} over $\calH_{II}$. We shall make use of the Euler-Lagrange equations \eqref{eq:coupled_problemIII_NEW} and some known results on regularity of solutions to local and nonlocal problems on Lipschitz domains in an iterative fashion. In first place, we review these results, and afterwards combine them to deduce the regularity of solutions under some additional assumptions on the domain configuration. Since \eqref{eq:coupled_problemIII_NEW} involves the interface operator $D_\Sigma$, which may not coincide with a classical normal derivative, we first consider the case in which $\Sigma = \emptyset$ and therefore one recovers two isolated Dirichlet problems. Afterwards, we show that if $s<1/2$ and $f$ is sufficiently smooth we can actually prove that $D_\Sigma$ corresponds to a classical Neumann operator and derive regularity of solutions in that case.

We present the results in this section for specific settings of particular interest to us. However, several variants are possible with minor adjustments to certain parameters, such as the type of domain (polygonal or arbitrary Lipschitz) or the regularity of the data. This is especially true for our choice of Proposition \ref{prop:local-regularity-mixed}, which could be replaced with any suitable regularity estimate applicable to problems with mixed Dirichlet/Neumann boundary conditions.

\subsection{Review of known results}
Here, we briefly discuss some global regularity estimates in the Besov scale for solutions of PDE on Lipschitz domains that we shall need in the sequel. For the sake of completeness, we collected some preliminary material on Besov spaces in Appendix \ref{sec:Besov}.

\subsubsection{Local operators}
There is a vast literature concerning regularity of local, linear,  elliptic PDE on bounded, non-smooth or piecewise smooth domains, cf. for example \cite{JeKe95, Savare98, Eb99, EbFr99, Jo99, DiKaRe15, ScSz22} and the books \cite{Grisvard, Costabel_book}.
In this work, depending on the configuration of the domain, we need to resort to either estimates with homogeneous Dirichlet conditions (in the case $\Sigma = \emptyset$) or with mixed Dirichlet and Neumann conditions with regions that may or may not intersect, depending on whether $\overline\Sigma \cap \overline{\partial \Omega}$ is empty or not. 
In particular, the way in which regions with Neumann and Dirichlet boundary conditions intersect critically affects the resulting regularity of solutions \cite{KaMa07}.

Besides the condition $\sigma_{\max}\ge\sl \ge\sigma_{\min}>0$ that we introduced in Section \ref{sec:l-nl-coupling}, we assume the diffusivity $\sl$ to be Lipschitz continuous, 
\begin{equation} \label{eq:Lipschitz-sigma}
|\sl(x) - \sl(y) | \lesssim |x-y| , \quad \forall x, y \in \Ol.
\end{equation}

For problems with homogeneous Dirichlet boundary conditions, we have the following result \cite[Theorem 2]{Savare98} (see also \cite{JeKe95} for related results in spaces with different integrability indices). Since for the sake of approximation we can only exploit minimal regularity, we shall conform ourselves with results of such a type.

\begin{proposition} \label{prop:local-regularity-dir}
Let $\domain$ be a bounded Lipschitz domain, $\sl$ satisfy \eqref{eq:Lipschitz-sigma}, let $f \in H^{-1}(\domain)$ be given and $u \in H^1_0(\domain)$ be the weak solution to the homogeneous Dirichlet problem
\[
 \left\lbrace \begin{aligned}
-\div (\sl \nabla u) = f  & \mbox{ in } \domain, \\
u = 0 & \mbox{ on } \partial\domain.
\end{aligned} \right.
\]
If $f\in B^{-1/2}_{2,1}(\domain)$, then $u \in \dot B^{3/2}_{2,\infty}(\domain)$, with 
\begin{equation*}\label{eq:local-bes-regularity-max}
  \| u \|_{\dot{B}^{3/2}_{2,\infty}(\domain)} \lesssim  \| f \|_{B^{-1/2}_{2,1}(\domain)}.
\end{equation*}
If $f\in H^{-1+\frac\theta{2}}(\domain)$ for some $\theta \in (0,1)$, then $u \in \widetilde{H}^{1+ \frac\theta{2}}(\domain)$ with
\begin{equation}\label{eq:local-sob-regularity-intermediate}
  \| u \|_{\widetilde{H}^{1+ \frac\theta{2}}(\domain)} \lesssim  \| f \|_{H^{-1+ \frac\theta{2}}(\domain)}.
\end{equation}
\end{proposition}

We next list some results regarding problems with mixed boundary conditions. Since our goal is to apply the resulting regularity estimates to derive a priori approximation rates for finite element discretizations of problem \eqref{eq:coupled_problemIII_NEW} on two-dimensional polygonal domains, we will only state results in unweighted, $L^2$-Sobolev spaces and for problems with a constant diffusivity, cf. \cite{Gr76}. Other related and important results in that could also be applied to our problem include regularity for curved polyhedra in 3d \cite{Da92}, or classical estimates on weighted spaces \cite{Ko67}.

\begin{proposition} \label{prop:local-regularity-mixed}
Let $\domain \subset \R^2$ be a bounded polygonal domain whose boundary is the union of open sets $\Gamma_D$, $\Gamma_N$, with $\Gamma_D \cap \Gamma_N = \emptyset$, $\Gamma_D \neq \emptyset$, and $\overline \Gamma_D \cap \overline\Gamma_N$ consisting of a finite set of points. Let $f \colon \domain \to \R$ be given and $u \in H^1(\domain)$ be the weak solution to the homogeneous Dirichlet problem
\[
 \left\lbrace \begin{aligned}
- \Delta u = f  & \mbox{ in } \domain, \\
u = 0 & \mbox{ on } \Gamma_D, \\
\frac{\partial u}{\partial \vn} = 0 & \mbox{ on } \Gamma_N.
\end{aligned} \right.
\]
Let $\mathcal{C} \subset \partial \Omega$ be the (finite) set of intersecting points between $\overline\Gamma_D$ and $\overline\Gamma_N$ and corner points in $\Gamma_D$ and $\Gamma_N$. For all $c \in \mathcal{C}$, let $\omega_c \in (0, 2 \pi)$ be the opening of the ($\domain$-interior) angle at $c$ and let
 \[
 \gamma_c = \left\lbrace \begin{array}{lr} 
 1 + \frac{\pi}{2\omega_c} & \mbox{ if } c \in \overline \Gamma_D \cap \overline\Gamma_N , \\
 1+ \frac{\pi}{\omega_c} & \mbox{ if } c \in \Gamma_D \cup \Gamma_N .\\
 \end{array} \right.
 \]
Then, for 
 \[\gamma = \gamma(\domain) =  \min\{\gamma_c : \,  c  \in \mathcal{C} \}>1, \]
if $f  \in H^{\gamma-2}(\domain)$ we have $u \in H^\gamma(\domain)$ and
\[
\| u \|_{H^\gamma(\domain)} \lesssim \| f \|_{H^{\gamma-2}(\domain)}.
\]

\end{proposition}

\subsubsection{Nonlocal operators}
Regularity results concerning nonlocal operators such as $(-\Delta_\snl)^s$ are much scarcer than their local counterparts. Unlike what is needed for the local operator, independently of whether $\Sigma = \emptyset$ or not, for the purposes of our paper we only require results involving Dirichlet boundary conditions, although of non-homogeneous type. 

At this point, we need to make some assumption on the diffusivity $\snl$. Following \cite[E1]{Savare98} and \cite[Hypothesis 2.11]{BoLiNo23}, to obtain extra regularity of solutions we require some assumptions about the smoothness of the diffusivities. In the setting of \cite{BoLiNo23}, we have $G(x,y,\rho) = \frac{C(n,s)}4 \snl(x,y) \rho^2$. Therefore, to apply the results in that paper, we only need to assume some H\"older continuity on $\snl$. More precisely, we assume that there exists $\beta \in (0,1]$ such that
\begin{equation} \label{eq:Holder-sigma}
| \snl(x,y) - \snl(x',y') | \lesssim \left(|x-x'|^\beta + |y-y'|^\beta \right).
\end{equation}
With that, we can apply \cite[Theorem 3.4]{BoLiNo23} to deduce regularity estimates for problems with homogeneous Dirichlet conditions. We briefly state a result in this regard. We refer to \cite[Corollary 1.1]{BoNo23} and \cite[Corollary 3.8]{BoLiNo23}.

\begin{proposition}
Let $s \in (0,1)$. Assume $\snl$ satisfies \eqref{eq:Holder-sigma}, and let $f\in H^{-s}(\domain)$ be given. Let $u \in \widetilde H^s(\domain)$ be the weak solution to the homogeneous Dirichlet problem
\begin{equation} \label{eq:Dirichlet-FL}
 \left\lbrace \begin{aligned}
(-\Delta_\snl)^s  u = f  & \mbox{ in } \domain, \\
u = 0 & \mbox{ in } \domain^c.
\end{aligned} \right.
\end{equation}
If $f\in B^{-s+\beta/2}_{2,1}(\domain)$, then $u \in \dot B^{s+\beta/2}_{2,\infty}(\domain)$, with 
\begin{equation}\label{eq:bes-regularity-max}
  \| u \|_{\dot{B}^{s+ \frac\beta{2}}_{2,\infty}(\domain)} \lesssim  \| f \|_{B^{-s+\frac\beta{2}}_{2,1}(\domain)}.
\end{equation}
If $f\in H^{-s+\lambda \frac\beta{2}}(\domain)$ for some $\lambda \in (0,1)$, then $u \in \widetilde{H}^{s+ \lambda \frac\beta{2}}(\domain)$ with
\begin{equation}\label{eq:sob-regularity-intermediate}
  \| u \|_{\widetilde{H}^{s+ \lambda \frac\beta{2}}(\domain)} \lesssim  \| f \|_{H^{-s+\lambda \frac\beta{2}}(\domain)}.
\end{equation}
\end{proposition}

\begin{remark}[optimality]
The regularity estimate \eqref{eq:bes-regularity-max} is the maximal one can expect even for a smooth right hand side $f$ and a smooth domain. As a matter of fact, the reduced regularity of solutions is caused by a rough boundary behavior, as illustrated by a simple example:
\[
\domain = B(0,1)\subset\R^n, \ f \equiv 1, \ \snl \equiv 1 \ \Rightarrow \ u(x) = c(n,s) (1-|x|^2)^s_+,
\]
where $c(n,s)= 2^{-2s}\Gamma(\frac{n}{2})\left(\Gamma(\frac{n+2s}{2})\Gamma(1+s)\right)^{-1}$.
Regularity of solutions to problem \eqref{eq:Dirichlet-FL} with a constant diffusivity has been thoroughly studied in recent years; in particular, reference \cite{RosOtonSerra} obtains sharp estimates on the boundary behavior of solutions, that leads to weighted H\"older regularity estimates. As a conclusion, the boundary expansion
\[
u(x) \simeq d(x, \partial\domain)^s \varphi(x),
\]
where $\varphi$ is smooth, turns out to be generic.
\end{remark}

To prove regularity results of solutions of problems with non-homogeneous exterior conditions, we only need to exploit the mapping properties of the operator $(-\Delta_\snl)^s$. For completeness, we provide a proof of the following lemma in detail.

\begin{lemma}\label{lemma:mapping-Laplacian}
Let $s \in (0,1/2)$, $\alpha \in (0, 1-2s)$, and the diffusivity $\snl$ satisfy \eqref{eq:Holder-sigma} with $\beta \ge \alpha$. Then, the weighted fractional Laplacian \eqref{eq:def-wifl} is a bounded operator from $H^{\alpha+2s} (\R^n)$ to $B^{\alpha}_{2,\infty}(\R^n).$
\end{lemma}
\begin{proof}
See Appendix \ref{sec:mapping-Laplacian}.
\end{proof}

We state the following theorem in a fashion that we will use in Section \ref{sec:regularity-small-s}. 

\begin{theorem} \label{thm:reg-nl-s-small}
Let $\domain$ be a bounded Lipschitz domain, $s \in (0,1/2)$, $g \in H^{\alpha+2s}(\domain^c)$ for some $\alpha \in (0, 1-2s)$, $\snl$ satisfy \eqref{eq:Holder-sigma} with $\beta \ge \alpha$, $f\colon \domain \to \R$, and $u \in H^s(\R^n)$ be the weak solution to the non-homogeneous Dirichlet problem
\[
 \left\lbrace \begin{aligned}
(-\Delta_\snl)^s  u = f  & \mbox{ in } \domain, \\
u = g & \mbox{ in } \domain^c.
\end{aligned} \right.
\]
If $f\in H^{-s+\theta \frac\beta{2}}(\domain)$ for some $\theta \in (0,1)$, then $u \in H^{s+\min\{ \alpha+s-\epsilon, \theta\frac\beta2 \}}(\R^n)$  for all $\epsilon > 0$ and
\begin{equation} \label{eq:max-reg-p-big}
\| u \|_{H^{s+\min\{ \alpha+s-\epsilon, \theta\frac\beta2 \}}(\R^n)} \lesssim  \| f  \|_{H^{-s+\theta\frac\beta2}(\domain)} + \| g \|_{H^{\alpha+2s} (\domain^c)}.
\end{equation}
The hidden constant above behaves as $\epsilon^{-1/2}$ if $\alpha+s \le \theta\frac\beta2$ and as $(\alpha+s-\theta\frac\beta2)^{-1}$ if $\alpha +s > \theta \frac\beta{2}$.
\end{theorem}
\begin{proof}
As  $\domain$ is a bounded Lipschitz domain and $g \in H^{\alpha+2s}(\domain^c)$, we extend it to $G \in H^{\alpha+2s}(\R^n)$. By using Lemma \ref{lemma:mapping-Laplacian}, we have $(-\Delta_\snl)^s G \in B^{\alpha}_{2,\infty}(\R^n)$ and then the function $\tilde u = u - G$ is a weak solution of the homogeneous Dirichlet problem
\[
 \left\lbrace \begin{aligned}
(-\Delta_\snl)^s \tilde u = \tilde f & \mbox{ in } \domain, \\
\tilde u = 0 & \mbox{ in } \domain^c,
\end{aligned} \right.
\]
with $\tilde f =  f  - (-\Delta_\snl)^s G \in B^{\alpha}_{2,\infty}(\domain)$. 

In case $\alpha \le -s + \theta \frac\beta{2}$, we use the embedding $B^{\alpha}_{2,\infty}(\domain) \subset H^{\alpha -\epsilon}(\domain)$ for some arbitrary $\epsilon > 0$ (see \eqref{eq:Besov-Sobolev-emb}), that gives rise to 
\[
\| \tilde f  \|_{H^{\alpha -\epsilon}(\domain)} \lesssim \| f  \|_{H^{-s + \theta \frac\beta{2}}(\domain)} + \epsilon^{-1/2} \| (-\Delta_\snl)^s G \|_{B^{\alpha}_{2,\infty}(\domain)},
\]
and use \eqref{eq:sob-regularity-intermediate} with $\lambda = \frac{2(\alpha+s-\epsilon)}\beta \in (0,1)$ to obtain
\[
  \| \tilde u \|_{\widetilde{H}^{\alpha +2s - \epsilon}(\domain)} \lesssim  \| f  \|_{H^{-s + \theta \frac\beta{2}}(\domain)} + \epsilon^{-1/2} \| g \|_{H^{\alpha+2s} (\domain^c)} .
\]

For $\alpha > -s + \theta \frac\beta2$ we again resort to \eqref{eq:Besov-Sobolev-emb}, exploit the embedding $B^{\alpha}_{2,\infty}(\domain) \subset H^{-s+\theta\frac\beta2}(\domain)$ 
with a constant $\mathcal{O}(\alpha+s-\theta\frac\beta2)$,
and  use \eqref{eq:sob-regularity-intermediate} with $\lambda = \theta$:
\[
  \| \tilde u \|_{\widetilde H^{s+\theta \frac\beta2}(\domain)} \lesssim  \| \tilde f  \|_{H^{-s+\theta\frac\beta2}(\domain)} \lesssim  
  \| f  \|_{H^{-s+\theta\frac\beta2}(\domain)} + \frac{ \| g \|_{H^{\alpha+2s} (\domain^c)}}{\alpha+s-\theta\frac\beta2} .
\]
\end{proof}

\begin{remark}[limitations on the value of $s$]
We proved Lemma \ref{lemma:mapping-Laplacian} and Theorem \ref{thm:reg-nl-s-small} for $s$ in the range $(0,1/2)$. The main reason behind this limitation becomes apparent upon inspection of our argument in Appendix \ref{sec:mapping-Laplacian}: we used first differences. For sufficiently smooth diffusivities, we expect the weighted fractional Laplacian to be an operator of order $2s$, but if $2s > 1$ then one cannot derive such a regularity by using first differences. If $\alpha \in (0, 2-2s)$ and $\snl$ is locally $C^\beta$ near the diagonal $\{x = y\}$ for some $\beta \ge \alpha$, then we expect $(-\Delta_\snl)^s \colon H^{\alpha+2s} (\R^n) \to B^{\alpha}_{2,\infty}(\R^n)$ to be a bounded operator.

However, we emphasize that we are interested in applying Theorem \ref{thm:reg-nl-s-small} in the case where $\Sigma \neq \emptyset$ and the operator $D_\Sigma$ coincides with the classical Neumann operator; as we already anticipated, and show in Section \ref{sec:regularity-small-s} below, we expect this only to happen provided $s<1/2$.
 \end{remark}

\subsection{A simple setting: isolated subdomains} \label{sec:isolated-subdomains}
From our discussion in Section \ref{subsec_Euler-Lagrange_III}, and in particular from the resulting Euler-Lagrange equations \eqref{eq:coupled_problemIII_NEW}, we note that  the problem simplifies notably in the case $\overline\Ol \cap \overline\Onl = \emptyset$. Indeed, in such a case, for all $x \in \Ol$ and $y \in \Onl$ we have $|x-y| \ge d(\Ol, \Onl) > 0$ and thus we have the pointwise bound
\[ \begin{split}
|\Nsnl^{\Onl} u (x)| & \le \frac{C(n,s) \, \| \snl \|_{L^\infty(\Ol \times \Onl)}}{d(\Ol, \Onl)^{n+2s}}  \int_{\Onl} |u(x)-u(y)| dy \\
& \lesssim  \left[ |u(x)| + \| u \|_{L^2(\Onl)} \right] 
\end{split} \]
for all $x \in \Ol$, with a hidden constant depending on $n,s, \snl, d(\Ol, \Onl),$ and $|\Onl|$.
Therefore, because $u \in L^2(\Omega)$ with $\| u \|_{L^2(\Omega)}\lesssim \| f \|_{L^2(\Omega)}$ --this is  because the energy space satisfies $\mathcal{H}_{II}\subset L^2(\Omega)$--, we immediately deduce
\[
\Nsnl^{\Onl} u \in L^2(\Ol) , \quad \mbox{with} \quad \| \Nsnl^{\Onl} u \|_{L^2(\Ol)} \lesssim  \| f \|_{L^2(\Omega)}. 
\]
Therefore, in this setting, in order to prove regularity of solutions within $\Ol$ it suffices to apply Proposition \ref{prop:local-regularity-dir}.
Additionally, in the subdomain $\Onl$ we can exploit the fact that $u|_{\Onl^c}$ vanishes on a neighborhood of $\Onl$. Namely, if we denote by $U$ the zero-extension of $u|_{\Onl^c}$ to $\Onl$, for $x \in \Onl$ we have
\[ \begin{split}
(-\Delta_\snl)^s U (x) & = C(n,s) \int_{\R^n} \snl(x,y) \frac{U(x) - U(y)}{|x-y|^{n+2s}} \, dy \\ & = - C(n,s) \int_{\Ol} \snl(x,y) \frac{U(y)}{|x-y|^{n+2s}} \, dy
\end{split} \]
and, because $d(\Ol, \Onl) > 0$, if $\snl$ satisfies \eqref{eq:Holder-sigma} we deduce $(-\Delta_\snl)^s U |_{\Onl} \in C^\beta(\overline\Onl) \subset H^{\beta-\epsilon}(\Onl)$ for all $\epsilon > 0$. We can therefore mimic the proof of Theorem \ref{thm:reg-nl-s-small} with this modification and arbitrary $s \in (0,1)$, and obtain enhanced regularity estimates. We summarize our discussion in the following result.

\begin{proposition}[isolated subdomains] \label{prop:reg-isolated}
Let $s \in (0,1)$.
Assume $\overline\Ol \cap \overline \Onl = \emptyset$, $\beta \in [0,1)$ with $\beta \ge 2s - 1$, $f \in B^{-s+\beta/2}_{2,1}(\Omega)$, and let $u \in \calH_{II}$ be the minimizer of \eqref{eq:energy-III}. Then, $u|_{\Ol} \in \dot B^{3/2}_{2,\infty}(\Ol)$ and $u|_{\Onl} \in \dot B^{s+\beta/2}_{2,\infty}(\Onl)$.
\end{proposition}
\begin{proof}
Let $u \in \calH_{II}$ be the minimizer of \eqref{eq:energy-III}. According to our discussion above, on $\Ol$ we have
\[
-\div (\sl \nabla u)  = f  - \Nsnl^{\Onl} u \in B^{-1/2}_{2,1}(\Ol), \quad u |_{\partial \Ol} = 0,
\]
and applying Proposition \ref{prop:local-regularity-dir} we deduce $u|_{\Ol} \in \dot B^{3/2}_{2,\infty}(\Ol)$. To prove the claimed regularity of $u|_{\Onl}$, we point out that the function $\tilde u = u - U$ satisfies the homogeneous Dirichlet problem
\[
(-\Delta_\snl)^s \tilde{u}  = f  - (-\Delta_\snl)^s U \in B^{-s+\beta/2}_{2,1}(\Onl), \quad \tilde{u} |_{\Onl^c} = 0,
\]
and we can apply \eqref{eq:bes-regularity-max}.
\end{proof}

\begin{remark}[generalizations] 
The assumption $\beta \ge 2s - 1$ in the previous proposition is by no means fundamental. We included it for the clarity of the statement. By the same argument, if $s>1/2$ one can also prove regularity in case $\beta < 2s - 1$, only that in such a case one would have a reduced regularity over $\Ol$. Similarly, we can also apply interpolation arguments to deduce regularity estimates in case $f$ is less regular than $B^{-s+\beta/2}_{2,1}(\Omega)$. 
\end{remark}

\subsection{Transmission problems -- the case $s < 1/2$} \label{sec:regularity-small-s}
We now assume that $s \in (0,1/2)$,  $\alpha \in (2s, s+1/2)$, and $f \in H^{\alpha - 2s}(\Onl) \subset L^2(\Onl)$.
As anticipated in Remark \ref{rem:homogeneous-Neumann}, because $s<1/2$ we can actually show that the solution $u$ satisfies a homogeneous Neumann boundary condition on the interface from $\Ol$.

\begin{lemma}
Let $\Omega$ be a bounded Lipschitz domain, $\snl$ satisfy \eqref{eq:Holder-sigma}, $u$ be a minimizer of the energy $E_{II}$ with $s \in (0, 1/2)$ and $f \in H^{-s+\theta\frac\beta2}(\Omega)$ for some $\theta \in (0,1)$ and satisfying $s < \theta\frac\beta2$. Then, $\Nsnl^{\Onl} u \big|_{\Ol}  \in L^2(\Ol)$.
\end{lemma}
\begin{proof}
The well-posedness of problem \eqref{eq:coupled_problemIII_NEW} implies $u|_{\Onl^c} \in H^1(\Onl^c)$, and because $u$ vanishes on $\Omega^c$, we have $u|_{\Onl^c} \in H^{1-\epsilon}(\Onl^c)$ for all $\epsilon > 0$. We can therefore apply Theorem \ref{thm:reg-nl-s-small} with $\alpha = 1 - 2s - \epsilon$; because $s < 1/2$ we can assume $\epsilon$ is small enough so that $\min\{ 1- s - \epsilon, \theta\frac\beta2 \} = \theta\frac\beta2$ and therefore we obtain the global regularity $u \in H^{s+\theta\frac\beta2}(\R^n)$.

Let us define $\omega \colon \Ol \to \R$,
\[
\omega (x) = \left(\int_{\Onl} \frac{1}{|x-y|^{n+2s- \theta\beta}} \, dy \right)^{1/2}.
\]
Then, because $ \theta\frac\beta2 > s$ and $\Ol$, $\Onl$ are bounded, an integration in polar coordinates allows us to deduce that $\omega \in L^\infty(\Ol)$ (with $ \| \omega \|_{L^\infty(\Ol)} \le C( n, s, \alpha, d(\Ol), d(\Onl) )$).

We can therefore compute, for all $x \in \Ol$ and by applying H\"older's inequality,
\[ \begin{split}
|\Nsnl^{\Onl} u (x)| & \le C(n,s) \, \| \snl \|_{L^\infty(\Ol \times \Onl)} \int_{\Onl} \frac{|u(x)-u(y)|}{|x-y|^{n+2s}} \, dy \\
& \le C(n,s) \, \| \snl \|_{L^\infty(\Ol \times \Onl)} \, \| \omega \|_{L^\infty(\Ol)} \left(\int_{\Onl} \frac{|u(x)-u(y)|^2}{|x-y|^{n+2s+\theta\beta}} \, dy \right)^{1/2}.
\end{split} \]
Squaring both sides above and integrating on $\Ol$ yields
\[
\| \Nsnl^{\Onl} u \|_{L^2(\Ol)}  
\lesssim \left(\int_{\Ol} \int_{\Onl} \frac{|u(x)-u(y)|^2}{|x-y|^{n+2s+\theta\beta}} \, dy \, dx \right)^{1/2}.
\]
The last integral above can be bounded by $\| u \|_{H^{s+\theta\frac\beta2}(\R^n)}$ and we conclude that  $\Nsnl^{\Onl} u \in L^2(\Ol)$. 
\end{proof}

Under the assumptions of the previous lemma, by (\ref{eq:normal-derivative}) we deduce that, on $\Ol$, $u$ solves the problem 
\[
 \left\lbrace \begin{aligned}
-\div (\sl \nabla u) = f - \Nsnl^{\Onl} u   & \mbox{ in } \Ol, \\
(\sl \frac{\partial u}{\partial \vn})|_{\Ol} = 0 & \mbox{ on } \Sigma, \\
u = 0 & \mbox{ on } \Gl = \partial\Ol \setminus \Sigma.
\end{aligned} \right.
\]
We have $f - \Nsnl^{\Onl} u \in L^2(\Ol)$. Next, we simplify our discussion by restricting $\Ol$ to be a two-dimensional, polygonal domain so that Proposition \ref{prop:local-regularity-mixed} can be applied.

\begin{proposition}[case $s<1/2$] \label{prop:regularity-small-s}
Let $\Omega \subset \R^2$ be a bounded polygonal domain such that $\Ol$ satisfies the assumptions in Proposition \ref{prop:local-regularity-mixed}. Let $f \in H^{-s+\theta\frac\beta2}(\Omega)$ for some $\theta \in (0,1)$ and satisfying $s < \theta\frac\beta2$, and $u$ be a minimizer of the energy $E_{II}$ with $s<1/2$. Then, $u|_{\Ol} \in H^{\gamma}(\Ol)$ for $\gamma = \min\{\gamma(\Ol),2\}$ and $u \in H^{s+\theta\frac\beta2}(\R^2)$.
\end{proposition}
\begin{proof}
Let $\gamma = \min\{\gamma(\Ol),2\} \in (1,2]$ be as in Proposition \ref{prop:local-regularity-mixed}. We have $f|_{\Ol}  \in L^2(\Ol)$ and therefore $u|_{\Ol} \in H^{\gamma}(\Ol)$. This readily implies $u|_{\Onl^c} \in H^{\gamma}(\Onl^c)$ and $u|_{\Onl^c} \in H^{1-\epsilon}(\Onl^c)$ for all $\epsilon > 0$. We can apply Theorem \ref{thm:reg-nl-s-small} to conclude that $u \in H^{s+\theta\frac\beta2}(\R^2)$.
\end{proof}

\begin{remark}[comments on the general case]
In this subsection, we have performed two different simplifications. On the one hand, we have restricted the analysis to problems on two spatial dimensions, and on polygonal domains. As we already stated, we made this assumption with the only goal of simplifying our analysis, and the proof we made could be easily adapted for other configurations in which different type of regularity estimates on $\Ol$ are available.

On the other hand, our analysis only covered the case $s<1/2$. The reason for this assumption is twofold. First, in our derivation of the strong form of the Euler-Lagrange equations we were able to show the transmission condition $D_\Sigma = 0$, and we actually expect the operator $D_\Sigma$ to coincide with a Neumann derivative only when $s<1/2$; it should be noted that, in the limit $s \to 1$ our problem recovers a standard transmission problem and therefore the interface condition we expect is like in \eqref{eq:local-transmission}. Second, the mapping properties of the operator $\Nsnl^{\Onl}$ are not clear when $s \ge 1/2$: while in principle we expect it to be an operator of order $2s \ge 1$, the lack of second differences in its definition difficults the analysis. In particular, we are not aware of a result like Theorem \ref{thm:trace-extension} for $t > 1$ and therefore we cannot give a definition like \eqref{eq:def-Nsnl} for functions in $\mathcal{H}^t(\Onl)$ with $t > 1$.

Nevertheless, a direct treatment of \eqref{eq:coupled_problemIII_NEW} as a transmission problem might allow to obtain regularity estimates in the whole range $s \in (0,1)$.
\end{remark}

\section{Numerical analysis and computational explorations} \label{sec:experiments}
In this section, we consider a finite element discretization of both proposed energies. We discuss their approximation capabilities and outline the main ingredients required in their analysis. We also include several numerical experiments, using the code developed in \cite{AcBeBo17} as a starting point, that illustrate the qualitative behavior of the proposed models as well as some of their most relevant features.

\subsection{Finite element setting, interpolation, and convergence}

We consider discretizations of the problems \eqref{eq:coupled_problemI_NEW} and \eqref{eq:coupled_problemIII_NEW} by means of the finite element method with piecewise linear continuous functions. 
Let $h_0 > 0$; for $h \in (0, h_0]$, we let $\mathcal{T}_h$ denote a mesh of $\Omega$, i.e., $\mathcal{T}_h = \{T\}$ is a partition of $\Omega$ into simplices $T$ of diameter $h_T$ and $h = \max_{T \in \Th} h_T$. Additionally, we assume that $\mathcal{T}_h$ meshes exactly both $\Ol$ and $\Onl$, and that the family $\{\Th \}_{h>0}$ is shape-regular, namely,
\begin{equation} \label{eq:shape-regularity}
\sigma := \sup_{h>0} \max_{T \in \Th} \frac{h_T}{\rho_T} <\infty,
\end{equation}
where $\rho_T $ is the diameter of the largest ball inscribed in $T$. We take simplices to be closed sets. We point out that our only assumption on meshes is shape-regularity; in particular, we do not assume any uniformity. In some cases, we choose meshes that are suitably graded towards the interface $\Sigma$ to better capture the solution behavior near that region.

Let $\overline{\mathcal{N}_h}$ be the set of vertices of $\Th$, $\mathcal{N}_h$ be the set of interior vertices, $N = \#\mathcal{N}_h$  and $\{ \varphi_i \}_{i=1}^N$ be the standard piecewise linear Lagrangian basis, with $\phii_i$ associated to the node $\x_i \in \mathcal{N}_h$. With this notation, the set of discrete functions is
\begin{equation*} \label{eq:FE_space}
\widetilde{\mathbb{V}}_h :=  \left\{v_h: \R^n \to \R \colon v_h \in C(\R^n), \ v_h = \sum_{i=1}^N v_i \varphi_i \right\},
 \end{equation*}
where $v_h$ is trivially extended by zero outside $\Omega$. It is clear that $\widetilde{\mathbb{V}}_h \subset \widetilde{H}^s(\Omega)$ for all $s \in (0,1)$. Therefore, we consider conforming finite element discretizations and seek $u_h \in \widetilde{\mathbb{V}}_h$ such that either
\begin{equation} \label{eq:Euler-Lagrange-E_I-discrete} \begin{aligned}
 \int_{\Ol} \sl \nabla u_h \cdot \nabla v_h\, dx + \frac{ C(n,s) }{2} \iint_{\Onl^2} \snl(x,y) \frac{(u_h(x)-u_h(y))(v_h(x)-v_h(y))}{|x-y|^{n+2s}} dy dx  \\
= \int_\Omega f v_h \, dx, \qquad \forall v_h \in \widetilde{\mathbb{V}}_h
\end{aligned}
\end{equation}
or 
\begin{equation} \label{eq:weak-EIIv3-discrete} \begin{aligned}
\int_{\Ol}  \sl \nabla u_h \cdot \nabla v_h\, dx + \frac{C(n,s)}{2}  \iint_{Q_{\Onl}} \snl(x,y) \frac{(u_h(x)-u_h(y))(v_h(x)-v_h(y))}{|x-y|^{n+2s}} dy dx\\ 
= \int_\Omega f v_h \, dx, \quad \forall v_h \in \widetilde{\mathbb{V}}_h
\end{aligned} \end{equation}
hold. These equations are simply discrete counterparts to \eqref{eq:Euler-Lagrange-E_I} and \eqref{eq:weak-EIIv3}, respectively.
Clearly, $u_h$ solves \eqref{eq:Euler-Lagrange-E_I-discrete} (resp. \eqref{eq:weak-EIIv3-discrete}) if and only if it is the minimizer of the restriction of the convex functional $E_I$ from \eqref{eq:energy-I} (resp. $E_{II}$ from \eqref{eq:energy-III}) over the linear space $\widetilde{\mathbb{V}}_h$; existence of discrete solutions follows immediately. 
Moreover, if we let $v_h = u_h$ in either \eqref{eq:Euler-Lagrange-E_I-discrete} or \eqref{eq:weak-EIIv3-discrete}, then we immediately obtain discrete stability bounds given $f \in L^2(\Omega)$ (or even in suitable negative-order Sobolev spaces).

Furthermore, as $\widetilde{\mathbb{V}}_h \subset  \calH_{I}$ (resp. $\widetilde{\mathbb{V}}_h \subset  \calH_{II}$) we trivially have a Galerkin orthogonality property and consequently a best approximation result holds. From this, one can combine suitable quasi-interpolation estimates with a density argument like \cite[Thm. 3.1.3]{RaviartThomas83} to prove the convergence of the approximations. In case one has at hand regularity estimates for the solutions, such as the ones described in Section \ref{sec:regularity}, convergence rates can be obtained. 
We do not delve into details, and only outline the main ingredients for the energy $E_{II}$. Using the fact that $\sl, \snl$ are bounded and uniformly positive, the best approximation result yields
\[
| u - u_h|_{H^1(\Ol)} + | u - u_h|_{V^s(\Onl | \R^n)} \lesssim \inf_{v_h \in \widetilde{\mathbb{V}_h}} | u - v_h|_{H^1(\Ol)} + | u - v_h|_{V^s(\Onl | \R^n)}.
\]
Thus, it suffices to use a suitable interpolation operator to bound the sum in the right-hand side above.
For clarity, we specifically consider the Scott-Zhang quasi-interpolation operator \cite{ScottZhang};  analogous conclusions could be drawn from interpolation using e.g. the Cl\'ement \cite{Clement} or Chen-Nochetto \cite{ChenNochetto} operators, although some modifications would be needed to deal with the interface.
Since in all our problems the solutions exhibit the minimal regularity $u \in H^\ell(\Omega)$ for some $\ell > 1/2$ for adequately smooth $\snl$ and $f$ (see Theorem~\ref{thm:reg-nl-s-small} with $\alpha$ close to $1-2s$), we can follow the original construction of the interpolation operator from \cite{ScottZhang}. 

For the following discussion, we consider an arbitrary bounded, polyhedral Lipschitz domain $\domain$ that is meshed exactly by a shape-regular family $\mathcal{T}_h = \{T\}$, and keep the notation introduced above. As we are considering piecewise linear Lagrange elements, for every node $\x_i \in \overline{\mathcal{N}_h}$ we consider an $(n-1)$-simplex $K_i$ such that $\x_i \in \partial K_i$, we design $I_h : \widetilde H^\ell(\domain) \to \widetilde{\mathbb{V}}_h$ by introducing the $L^2$-dual basis of the nodal basis over $K_i$ and, in particular, we let $\psi_i$ be the basis element such that $\int_{K_i} \varphi_i \psi_i = 1$.  With this, we define
\begin{equation} \label{eq:def-quasi-interpolation}
  I_h v = \sum_{i \colon \x_i \in \overline{\mathcal{N}_h}} \left( \int_{K_i} v(x) \psi_i(x)  dx \right) \varphi_i.
\end{equation}
 
Even if the sum is taken over all the mesh nodes, for functions that vanish on $\partial \domain$ we have $I_h v \big|_{\partial \Omega} = 0$. For the analysis of the operator, we need the following notions: given $T \in \Th$, we respectively define its first and second star by
\[
  S^1_T := \bigcup \left\{ T' \in \Th: \, T \cap T' \neq \emptyset \right\}, \quad
  S^2_T := \bigcup \left\{ T' \in \Th: \, S^1_T \cap T' \neq \emptyset \right\}.
\]
We also require a slight modification to deal with boundary-touching simplices: we define
\begin{equation*}
\widetilde{S}_T^1 :=
\left\lbrace \begin{array}{rl}
  S_T^1   & \textrm{if } T \cap \partial\Omega = \emptyset,
  \\
  B_T    & \textrm{otherwise,}
\end{array}\right.
\end{equation*}
where $B_T$ is a ball centered at the barycenter of $T$, with radius comparable to $h_T$, and such that $S_T^1\subset B_T$. We also consider
\[
\widetilde{S}_T^2 := \bigcup \big\{ \widetilde{S}_{T'}^1 \colon \, T' \in \Th, \, \mbox{int}(T') \cap S_T^1 \neq \emptyset \big\}.
\]

In \cite[Section 4]{ScottZhang}, approximation properties of this operator in integer-order Sobolev spaces are studied. Combining the results therein with interpolation between Sobolev spaces, we deduce the following
\[
\| v - I_h v \|_{L^2(T)} + h_T \| \nabla (v - I_h v) \|_{L^2(T)} \lesssim h_T^r |v|_{H^r(S^1_T)} \quad \forall v \in H^r(S^1_T), \ r \in [1, 2],
\]
with $\lesssim$ independent of $h$ and $T\in\Th$.
For less-regular functions, the approximation properties of the Scott-Zhang operator were first studied in \cite{Ciarlet}. Here, we recall a related estimate derived in \cite{AcBo17,BoNo23-2}, that is tailored to the analysis of our problem: for all  $v \in H^r(\widetilde{S}_T^2)$, $s \in (0,1)$, $r \in [s,2]$, it holds that
\[
\left(\iint_{T \times \widetilde{S}_T^1}  \frac{|(v-I_h v) (x) - (v-I_h v) (y)|^2}{|x-y|^{n+2s}} \, dy \, dx \right)^{1/2} \lesssim h_T^{r-s} |v|_{H^r(\widetilde{S}_T^2)}.
\]

To put together the local approximability estimates with respect to fractional-order seminorms, one needs to use a suitable localization. The first results along this line were obtained in \cite{Faermann2, Faermann}, where localization of the $H^s(\domain)$-seminorms are studied; a slight modification needs to be introduced to deal with $\widetilde H^s(\domain)$-seminorms, see \cite[Section 4.1]{BoNo23-2} for details. Concretely, we have
\begin{equation} \label{eq:localization-F-tilde}
 \|v\|_{\widetilde H^s(\domain)}^2 \le \frac{C(n,s)}2 \sum_{T \in \Th} \left[ \iint_{T \times \widetilde{S}_T^1}\!  \frac{|v (x) - v (y)|^2}{|x-y |^{n+2s}} \; dy \; dx + \frac{ c(n,\sigma) \| v \|^2_{L^2(T)}}{s h_T^{2s}} \right],
\end{equation}
where $\sigma$ is the shape-regularity constant introduced in \eqref{eq:shape-regularity}. Collecting the results mentioned above, we obtain the following global interpolation estimates.

\begin{proposition} Let $\domain$ be a bounded, polyhedral Lipschitz domain, $s \in (0,1)$, and $I_h$ be the Scott-Zhang interpolation operator defined above on a shape regular family of meshes of $\domain$. Then, we have
\begin{equation} \label{eq:SZ-H1}
 |v - I_h v |_{H^1(\domain)} \lesssim \left(\sum_{T \in \Th} h_T^{2(r-1)} |v|_{H^r(S^1_T)}^2 \right)^{1/2} \quad \forall v \in H^r(\domain), \ r \in [1, 2],
\end{equation}
and
\begin{equation} \label{eq:SZ-Hs}
 |v - I_h v |_{\widetilde H^s(\domain)} \lesssim \left(\sum_{T \in \Th} h_T^{2(r-s)} |v|_{H^r(\widetilde S^2_T)}^2 \right)^{1/2} \quad \forall v \in \widetilde H^r(\domain), \ r \in [s, 2].
\end{equation}\end{proposition}

Under the assumption that $\Th$ meshes exactly both $\Ol$ and $\Onl$, the result above immediately allows one to deduce convergence rates for our energy $E_{II}$ provided we have at hand some estimates on the regularity of solutions. Indeed, it is clear that \eqref{eq:SZ-H1} yields convergence rates in $H^1(\Ol)$ for the quasi-interpolation operator. Additionally, even in problems in which $\Sigma \neq \emptyset$, the seminorm \eqref{eq:def-Vs-norm} can be trivially bounded in terms of the global $\widetilde H^s(\Omega)$ norm, since $Q_{\Onl} \subset Q_\Omega$. Thus, we can bound $|u - I_h u|_{V^s(\Onl | \R^n)} \le |u - I_h u|_{\widetilde H^s(\Omega)}$ and we can use \eqref{eq:SZ-Hs}. In the settings of either Sections \ref{sec:isolated-subdomains} or \ref{sec:regularity-small-s}, the resulting convergence rates read as follows.

\begin{proposition}[convergence rates] \label{prop:conv-rates} Let  $\{\Th \}_{h>0}$ be a shape-regular, conforming family of meshes in $\Omega$, such that for every $h>0$ the meshes $\Th$ cover exactly both $\Ol$ and $\Onl$, and let  $h = \max_{T \in \Th} h_T$.

In the setting of Proposition \ref{prop:reg-isolated}, we have
\[
| u - u_h|_{H^1(\Ol)} + | u - u_h|_{V^s(\Onl | \R^n)} \lesssim h^{\frac\beta2} \| f \|_{B^{-s+\beta/2}_{2,1}(\Omega)}
\]
Additionally, under the same hypotheses as in Proposition \ref{prop:regularity-small-s}, we have
\[
| u - u_h|_{H^1(\Ol)} + | u - u_h|_{V^s(\Onl | \R^n)} \lesssim h^{\min\{\gamma, \, \theta\frac\beta2\}} \| f \|_{ H^{-s+\theta\frac\beta2}(\Omega)}.
\]
\end{proposition}

\subsection{Experiments}
In this section we present the results to several computational experiments involving the coupled models we proposed in Section \ref{sec:l-nl-coupling}.
We mainly focus on \eqref{eq:coupled_problemIII_NEW}, but we also show some experiments related to \eqref{eq:coupled_problemI_NEW}. In general, explicit solutions to the problems are not available, and therefore our purpose in this section is, rather than numerically corroborating the validity of Proposition \ref{prop:conv-rates}, to explore some salient features of this problem. 
Concretely, we investigate the behavior of solutions near the interface of the domain and provide evidence that the interface operator $ D_\Sigma $ from \eqref{eq:coupled_problemIII_NEW} coincides with the normal derivative for $s < 1/2$. We illustrate that, even when $ \Omega $ consists of two disconnected subdomains, there is effective coupling between the models. We also present numerical evidence of the development of singularities on $\Ol $, caused by the way $\Sigma$ intersects $\partial \Omega$, even when $\Ol$ is a convex domain. Furthermore, we study the eigenvalues of the discrete problems for a fixed mesh size and show numerically that \eqref{eq:coupled_problemIII_NEW} recovers a standard local transmission problem in the limit as $ s \to 1 $. Finally, we compare the behavior of the two models \eqref{eq:coupled_problemI_NEW} and \eqref{eq:coupled_problemIII_NEW} in this limit.

\subsubsection{Interface behavior} \label{sec:two-balls-experiment}
As discussed in Section \ref{sec:regularity-small-s}, a relevant feature of solutions to \eqref{eq:coupled_problemIII_NEW} is that, for $s<1/2$, there is a homogeneous Neumann interface condition from the local subdomain. In contrast, for $s \to 1$ we formally expect to recover a standard transmission condition. We illustrate this behavior with the following experiment.
Let $\Omega = B(0,2)$, with $\Onl = B(0,1)$, $\Ol = B(0,2) \setminus \overline{B(0,1)}$, $\sl \equiv 1$, $\snl \equiv 1$, $f \equiv 1$. Formally, for $s=1$ we recover the Dirichlet problem $u \in H^1_0(\Omega)$: $-\Delta u = 1$, whose solution is $u(x_1,x_2) = \frac{(4-(x_1^2+x_2^2))_+}4$ and satisfies ${\frac{\partial u}{\partial \vn}}|_{\Ol} \equiv -\frac12$ at the interface $\Sigma = \partial B(0,1)$.

We compute the normal derivative on the interface of the solutions to \eqref{eq:weak-EIIv3-discrete} for $s = \frac{j}{100}$, $j = 1, \ldots, 99$. To better capture the solution behavior, we consider approximations on adapted meshes in the spirit of Grisvard \cite[Section 8.4]{Grisvard} grading according to the distance to the interface. More precisely, we fix $h>0$ and $\mu \ge 1$ and build meshes such that the triangles have size
\begin{equation} \label{eq:grading} h_T \simeq \begin{cases}
h \,d(T,\Sigma)^{(\mu-1)/\mu} &  \mbox{ if } T\cap \Sigma = \emptyset, \\
h^{\mu}  & \mbox{ if } T\cap \Sigma \neq \emptyset.
 \end{cases} \end{equation}
 
Figure \ref{fig:interface-two-balls} exhibits different slices of the computed solutions, as well as the computed value of the normal derivative at the point $(1,0)$, obtained on a mesh graded as above with $\mu = 2.5$ and with $26077$ elements on $\Ol$ and $40306$ on $\Onl$. From the right panel, it is apparent that $\left|\frac{\partial}{\partial \vn}(u_h|_{\Omega_1})(1,0)\right|$ increases with respect to $s$. While it is clear that none of the computed values equals zero, we point out that for $s=0.01$ and $s=0.5$ we obtained $\frac{\partial}{\partial \vn}(u_h|_{\Omega_1})(1,0) \approx -6.7\times 10^{-4}$  and $\frac{\partial}{\partial \vn}(u_h|_{\Omega_1})(1,0) \approx -9.4\times 10^{-3}$, respectively. In contrast, for $s=0.99$ we computed $\frac{\partial}{\partial \vn}(u_h|_{\Omega_1})(1,0) \approx -0.430$; to better illustrate the convergence for $s \to 1$, we also run the experiment with $s = 0.999$ and obtained $\frac{\partial}{\partial \vn}(u_h|_{\Omega_1})(1,0) \approx -0.494$, in very good agreement with the expected limit value $-1/2$. Finally, we point out that, in our experiments we observed that, for all $s$, $\frac{\partial}{\partial \vn}(u_h|_{\Omega_1})(1,0)$ was monotonically increasing as the meshes were increasingly refined.

\begin{figure}[h!] 
\begin{center}
   \includegraphics[width=.45\textwidth,keepaspectratio=true]{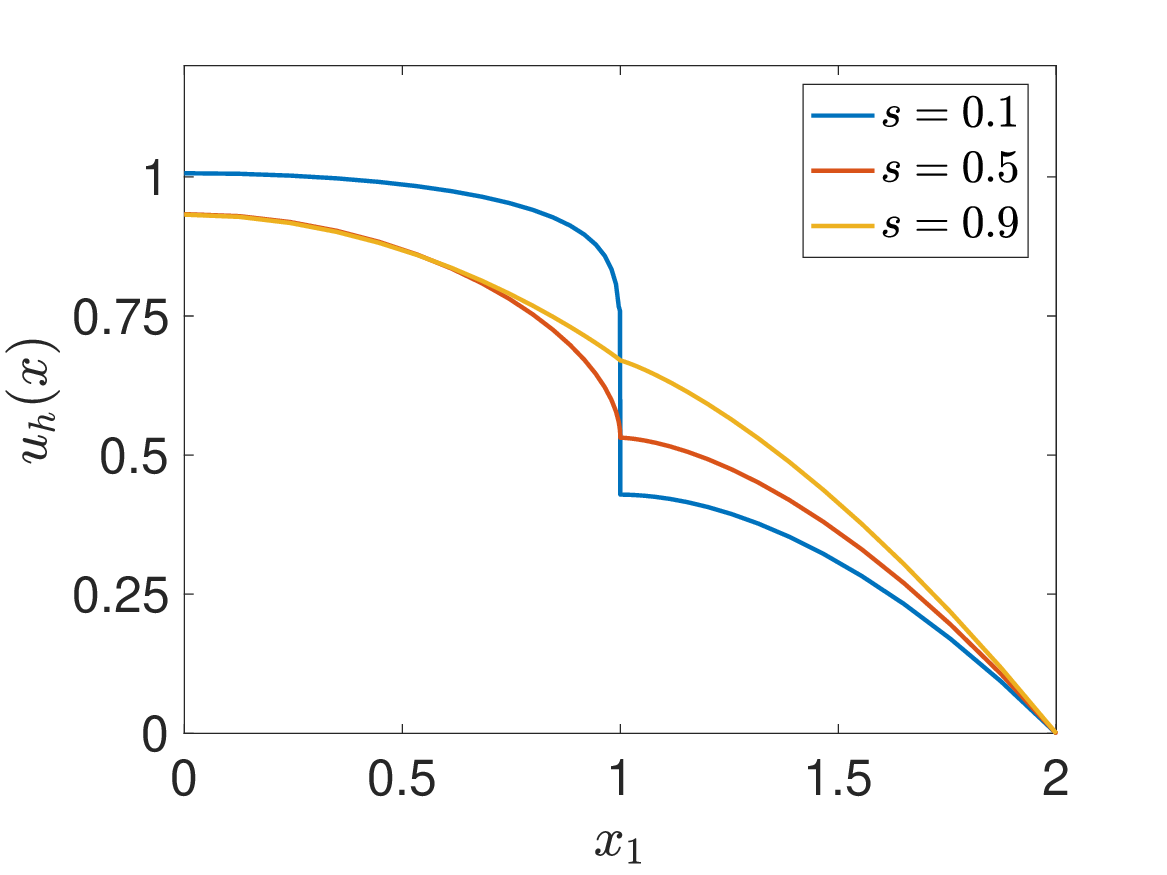}
   \includegraphics[width=.45\textwidth,keepaspectratio=true]{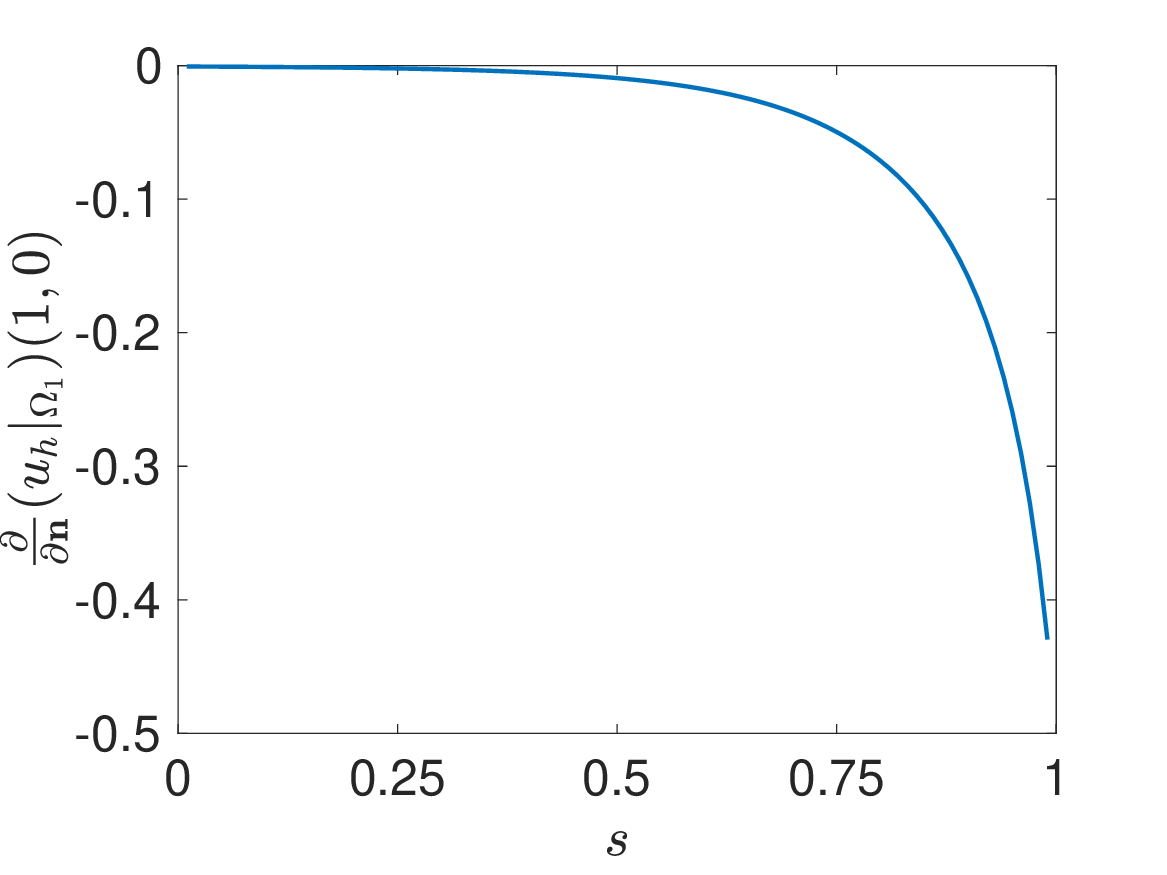}
	\caption{Finite element solutions to the problem described in Section \ref{sec:two-balls-experiment}. Left panel: slices $\{ x_1 > 0 ,x_2 = 0\}$ for $s=0.1, 0.5, 0.9$. Right panel: computed values of $\frac{\partial}{\partial \vn}(u_h|_{\Omega_1})(1,0)$ for different values of $s$.}  	\label{fig:interface-two-balls} 
 \end{center}
\end{figure}
\subsubsection{Isolated subdomains} 
Our next set of experiments aims to illustrate the coupling between the local and nonlocal models for the energy \eqref{eq:energy-III} even when $\Omega$ is composed of two isolated subdomains, namely $\Sigma = \emptyset$.
We consider $\Ol = (-\frac{3}4, -\frac14) \times (-\frac12,\frac12)$, $\Onl = (\frac14, -\frac34) \times (-\frac12,\frac12)$; $s = 0.5$; $\sl \equiv 1$, $\snl \equiv 1$; and either $f = \chi_{\Ol}$ or $f = \chi_{\Onl}$. Figure \ref{fig:two_subdomains} shows that finite element solutions we obtained in this setting and indicates that, even if $f$ is identically zero on one subdomain in each case, the solution does not vanish. On each panel, we use different color maps for each subdomain so that the behavior of the solution becomes clearer.

\begin{figure}[h!] 
\begin{center} \begin{tabular}{cc|cc}
   \includegraphics[width=.4\textwidth,keepaspectratio=true]{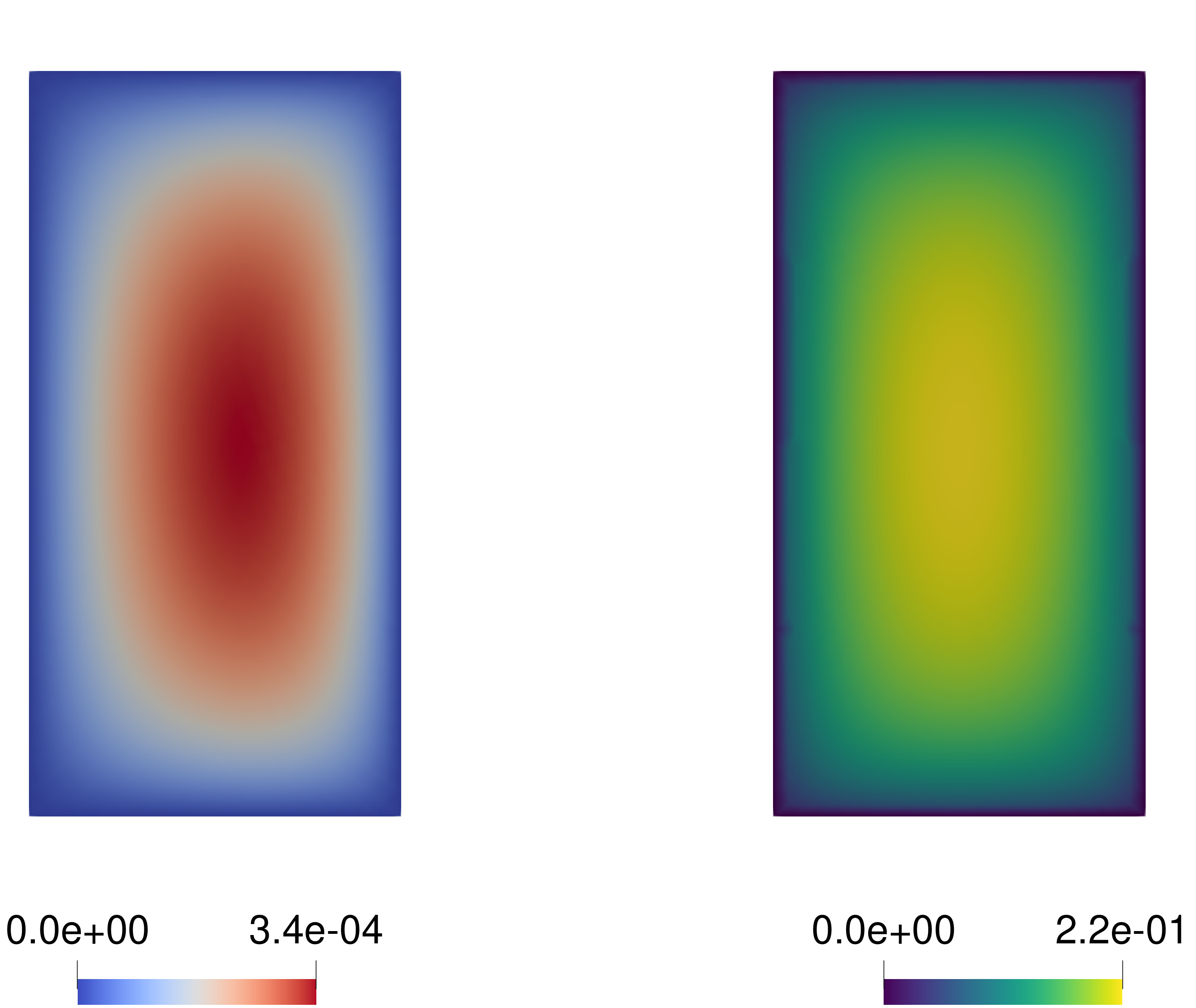}& \hspace{.04\textwidth}  & \hspace{.04\textwidth}  &
   \includegraphics[width=.4\textwidth,keepaspectratio=true]{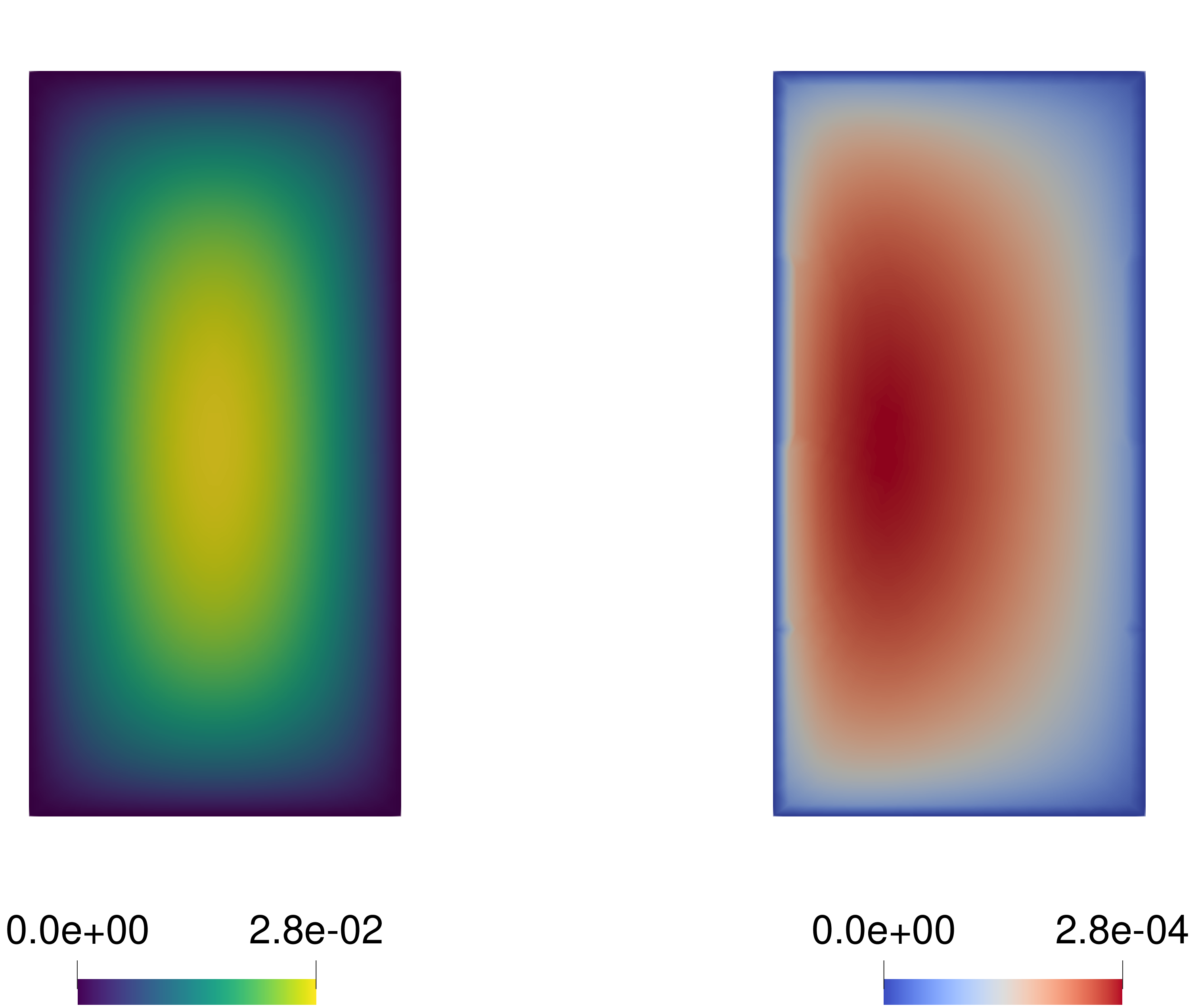}
   \end{tabular}
	\caption{Minimizers of \eqref{eq:weak-EIIv3-discrete} in the setting of Section \ref{sec:isolated-subdomains}. Left panel: $f = \chi_{\Ol}$; right panel: $f = \chi_{\Onl}$.} 	\label{fig:two_subdomains} 
 \end{center}
\end{figure}

\subsubsection{Singularities arising in the local model} \label{sec:singularities-experiment}
Our next experiment explores the development of algebraic singularities over $\Ol$ for minimizers to \eqref{eq:energy-III} and their dependence on $s$.
We consider $\Omega = (-1/2,1/2)^2 \setminus [-1/2,0]^2$, with $\Ol = \Omega \cap \{ x_2 > 0 \}$, $\sl \equiv 1$, $\snl \equiv 1$, and $f \equiv 1$. We observe that the domain configuration is such that $\Sigma$ intersects $\partial \Omega$ with a straight angle at the origin. Thus, according to Propositions \ref{prop:local-regularity-mixed} and \ref{prop:regularity-small-s}, in this setting, for $s < 1/2$, the regularity pickup in $\Ol$ is given by $\gamma(\Ol) = 3/2$. On the other hand, as $s \to 1$ one formally recovers a classical Poisson problem on the $L$-shaped $\Omega$ with homogeneous Dirichlet boundary conditions, that possesses a regularity pickup index $\gamma(\Omega) = 5/3$.

We aim to illustrate the dependence of the solution regularity for this problem with respect to $s$. For this purpose, we used a mesh graded similarly to \eqref{eq:grading}, although we considered the distance to the origin instead of the distance to the interface to drive the grading. The left panel in Figure \ref{fig:slices_s01} depicts the mesh in $\Omega$: we have a local mesh size near the origin $h \approx 1.5 \times 10^{-5}$. 

 \begin{figure}[h!] 
\begin{center} \begin{tabular}{ccccc}
 \includegraphics[width=.325\textwidth,keepaspectratio=true]{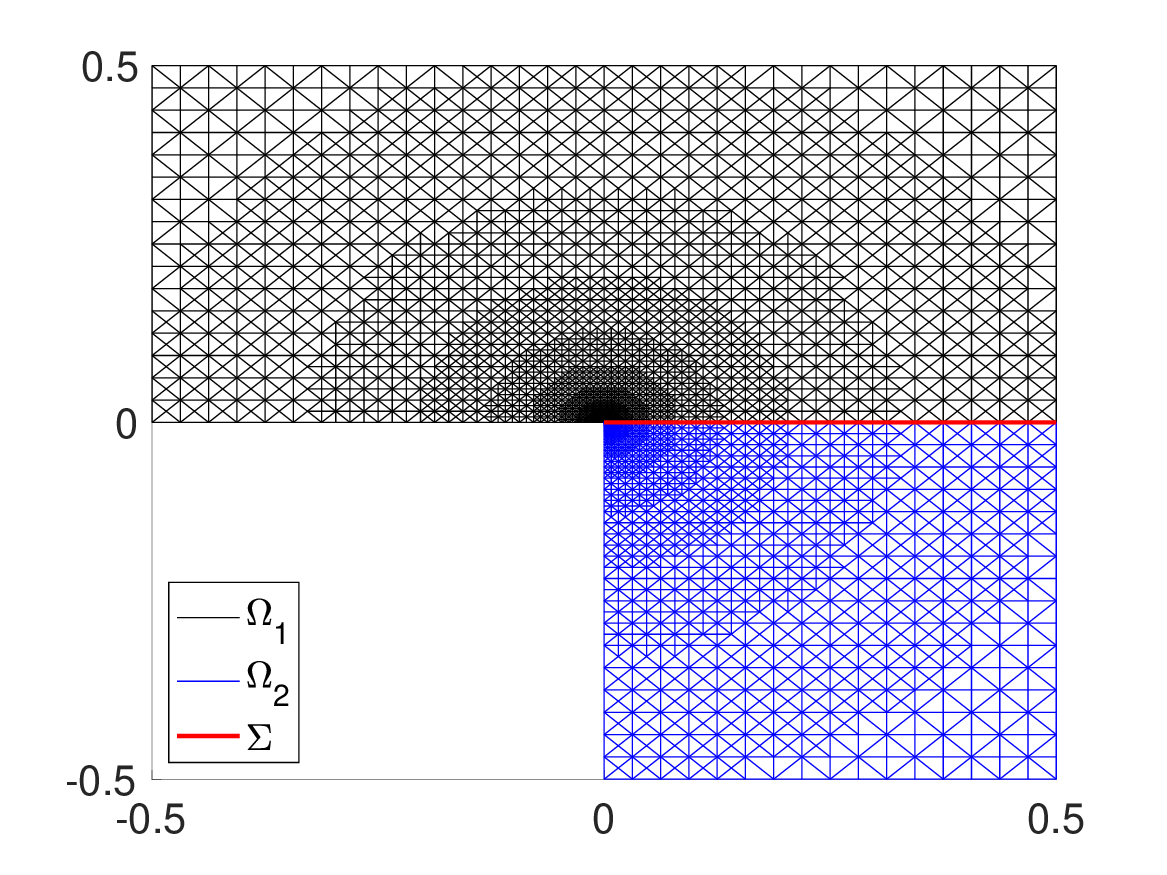}  & \hspace{-0.75cm}  & \includegraphics[width=.325\textwidth,keepaspectratio=true]{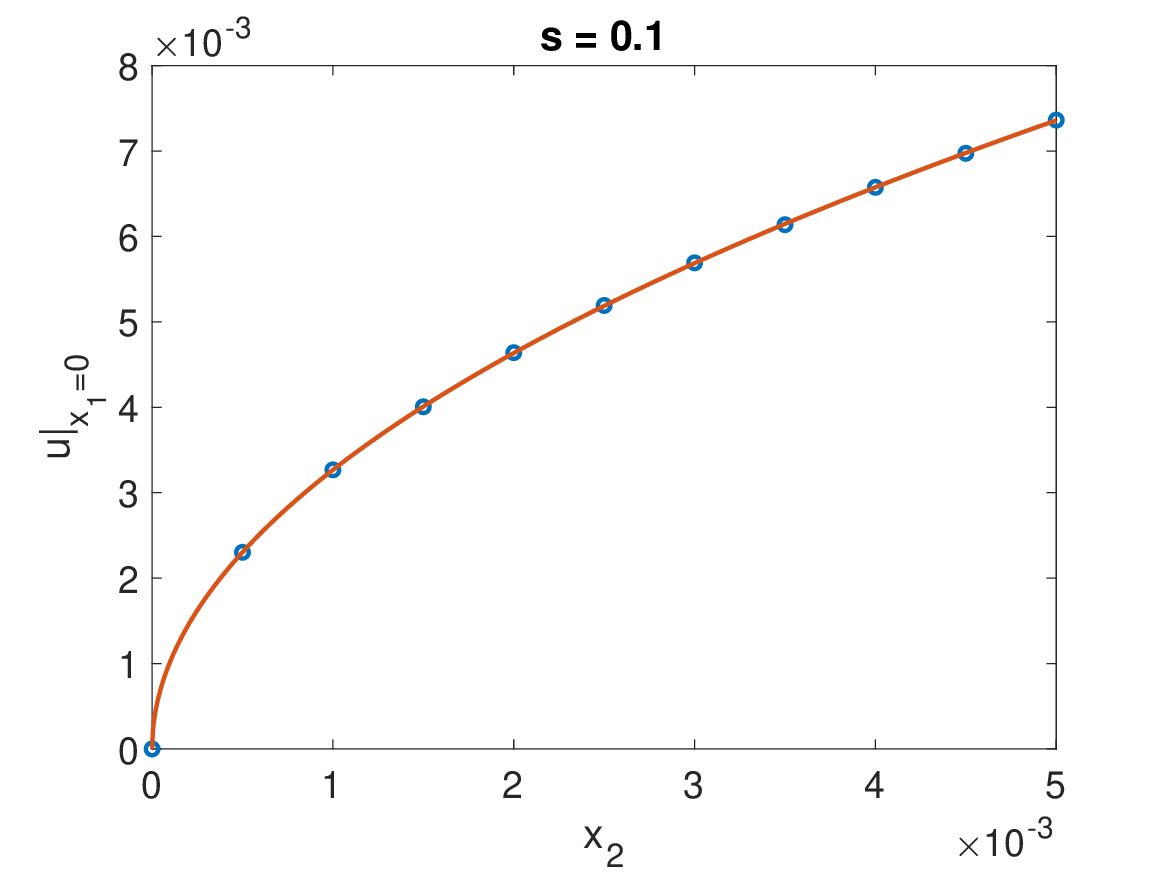}   & \hspace{-0.75cm} & 
   \includegraphics[width=.325\textwidth,keepaspectratio=true]{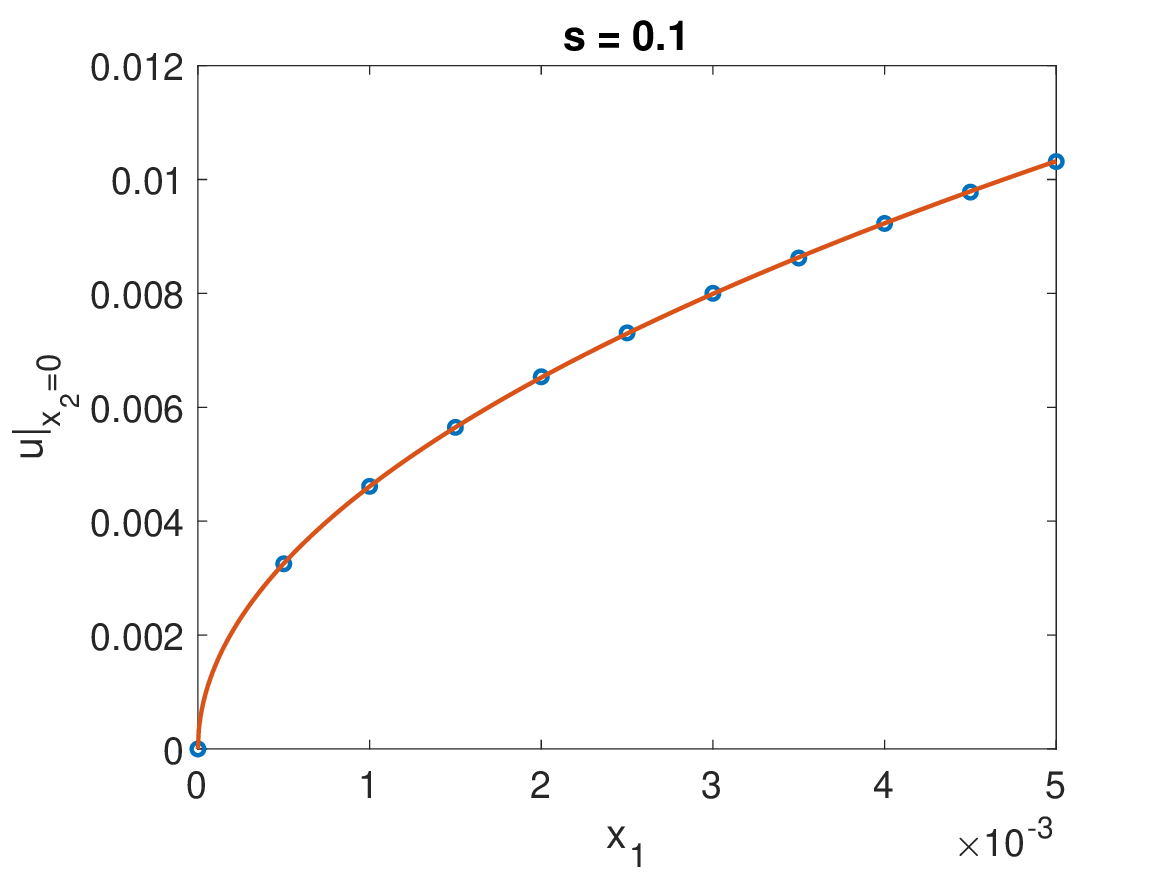}
   \end{tabular}
 \end{center}
\caption{Singular behavior in $\Ol$ of solution for $s=0.1$. Left: top view; center: slice $x_1=0$; right: slice $x_2=0$.} 	\label{fig:slices_s01}
\end{figure}
 We performed least-squares fittings of our computed solution over the segments $\{0\} \times (0,1/200)$ and $(0,1/200) \times  \{0\}$ with a model of the form $u(t) \simeq C t^\gamma$ using $11$ equally spaced sampling points, as shown in the center and right panels of Figure \ref{fig:slices_s01} for $s=0.1$; our findings are summarized in Table \ref{tab:least-squares-graded}. 
 We observe that, for small $s$ and even up to $s \approx 0.7$, the asymptotic behavior of $u|_{\Ol}$ near the origin is consistent with the value $\gamma(\Ol) = 3/2$, there is a transition for larger values of $s$, and for $s\to 1$ one recovers an exponent consistent with the expected global regularity $\gamma(\Omega) = 5/3$.

\begin{table}[h!]
\begin{center}
\begin{tabular}{c|c|c|c|c|c|c|} 
\cline{2-7}
 & $s = 0.1$  &  $s = 0.2$ & $s = 0.3$  &  $s = 0.4$ & $s = 0.5$  & $s = 0.6$\\ \hline
\multicolumn{1}{|c|}{$x_1=0$} & $0.5046$ & $0.5041$ & $0.5037$ & $0.5032$ & $0.5028$ & $0.5028$\\
\hline
\multicolumn{1}{|c|}{$x_2=0$} & $0.5009$ & $0.5012$ & $0.5015$ & $0.5017$ & $0.5018$ & $0.5017$\\
\hline
\multicolumn{1}{c}{ \ }  \\
\cline{2-7}
 & $s = 0.7$ & $s = 0.8$ & $s = 0.9$ & $s=0.95$  & $s=0.99$ & $s=0.999$ \\ \hline
\multicolumn{1}{|c|}{$x_1=0$}  & $0.5046$ & $0.5142$ & $0.5544$ & $0.6007$ & $0.6527$ & $0.6655$\\
\hline
\multicolumn{1}{|c|}{$x_2=0$}   & $0.5016$ & $0.5055$ & $0.5382$ & $0.5862$ & $0.6482$ & $0.6650$\\
\hline
\end{tabular}
	\caption{Computed singular exponents via least-squares fitting on the segments $\{0\} \times (0,1/200)$ and $(0,1/200) \times  \{0\}$  for different values of $s$ on a graded mesh.} 	\label{tab:least-squares-graded}
 \end{center}
\end{table}

\subsubsection{Eigenvalue comparison} \label{sec:eigenvalues}

For the next experiment we add a negative, zero-order term to the energy $E_{II}$. More precisely, for $\omega > 0$ we consider
$E_\omega (v) := E_{II}(v) - \omega^2 \| v \|_{L^2(\Omega)}^2$. Clearly, the energy minimization problem becomes ill-posed provided $\omega^2$ coincides with an eigenvalue of the associated differential operator over $\Omega$ with homogeneous boundary conditions. We test our code in the same setting as in \S\ref{sec:two-balls-experiment}: since $\sl \equiv 1$ and $\snl \equiv 1$, formally, for $s \sim 1$ we expect to lose uniqueness of solutions as $\omega^2$ approaches any eigenvalue of the (classical) Laplacian with homogeneous Dirichlet boundary conditions. Over $\Omega = B(0,2) \subset \R^2$, such eigenvalues are given by
\[
\lambda_{m,n} = \frac{j_{m,n}^2}4,
\]
where $j_{m,n}$ is the $n$-th positive  zero of the Bessel function of the first kind $J_m$. In increasing order, the first four numerical values are
\[
\lambda_{0,1} \approx 1.45,\quad \lambda_{1,1} \approx 3.67, \quad \lambda_{2,1} \approx 6.59, \quad \lambda_{0,2} \approx 7.62.
\]
Figure \ref{fig:conditioning-non-coercive} reports the estimated $\ell_2$-condition numbers of the resulting system matrices on quasi-uniform meshes with $3442$ elements on $\Ol$ and 2249 on $\Onl$ --this corresponds to a mesh size $h \approx 5\times10^{-2}$-- as a function of $\omega^2$ and for different values of $s$ and $\omega^2$ spanning $(0,10)$. For $s=0.99$, we observe a very good agreement with the corresponding eigenvalues of the limit local problem while as $s$ decreases we observe that the spectrum tends to shift to the left. As $s \to 0$, the eigenvalues cluster around~1 from above, see the close-up for  $\omega^2 \in (0,4)$ and $s=0.25$.
 \begin{figure}[h!] 
\begin{center} \begin{tabular}{ccccc}
 \includegraphics[width=.325\textwidth,keepaspectratio=true]{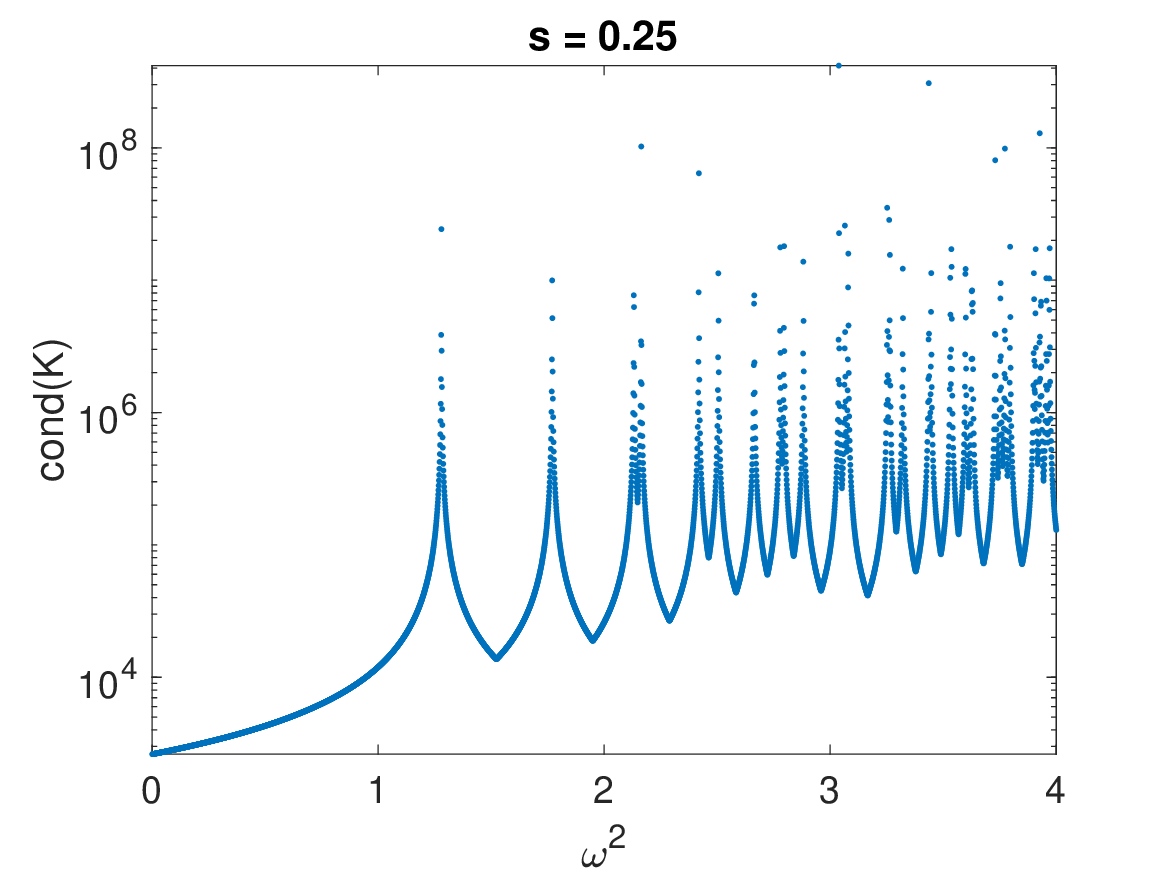} & \hspace{-0.75cm} & 
   \includegraphics[width=.325\textwidth,keepaspectratio=true]{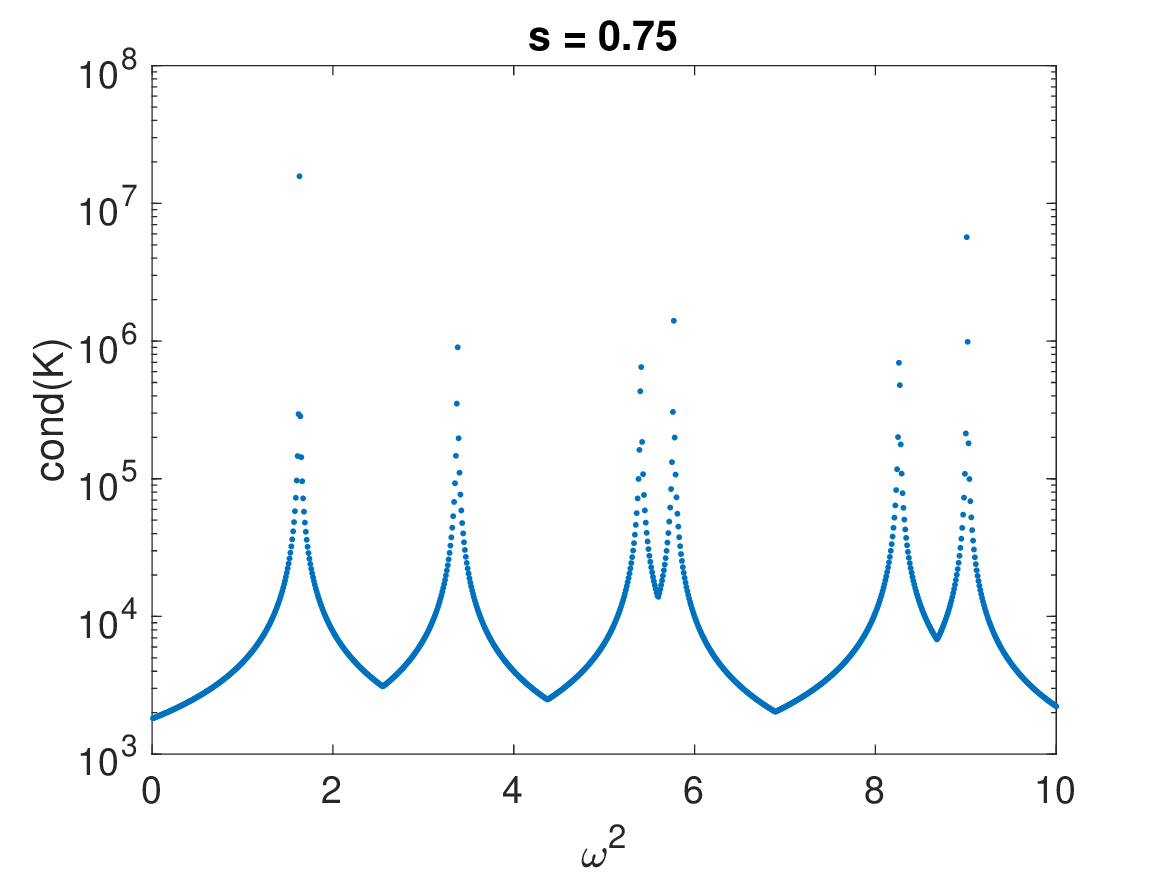}  & \hspace{-0.75cm} &
   \includegraphics[width=.325\textwidth,keepaspectratio=true]{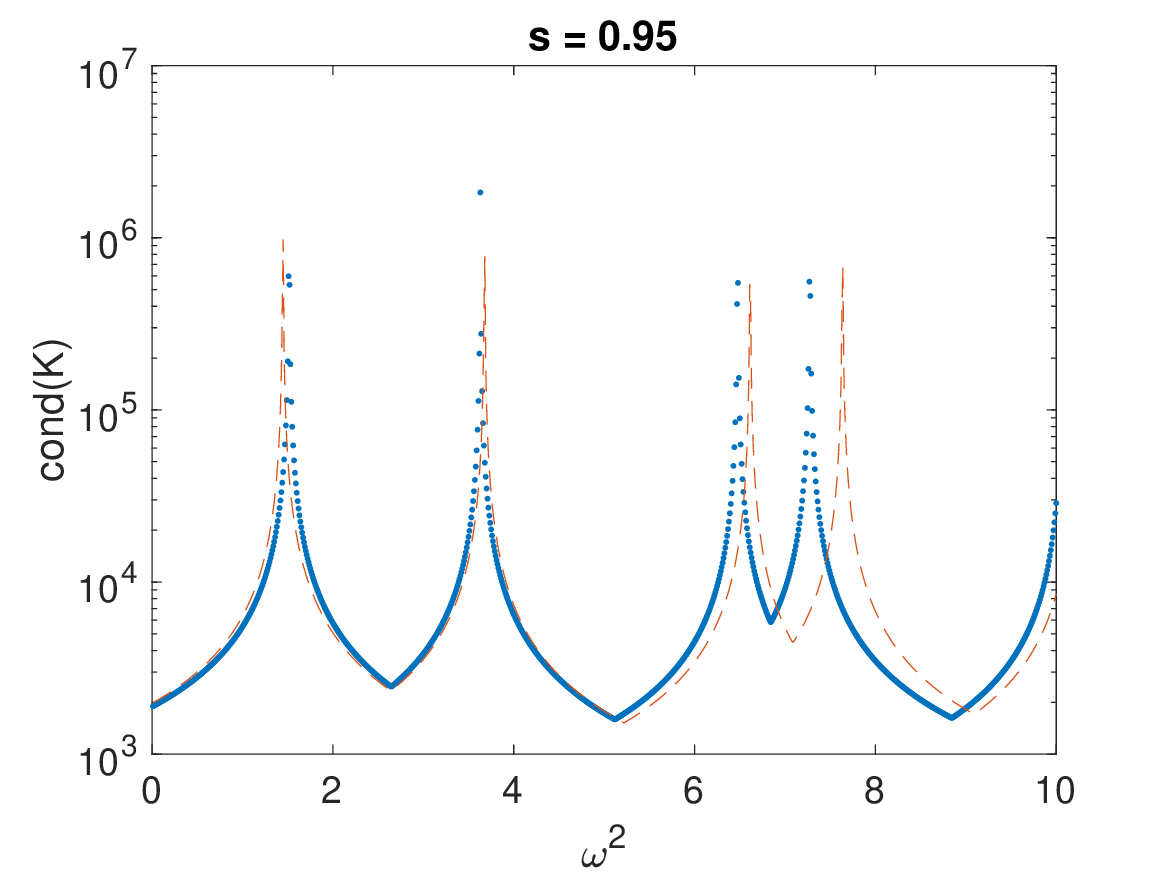}
   \end{tabular}
	\caption{In blue, computed condition numbers in the setting of \S\ref{sec:eigenvalues} for $s=0.25$ (left), $s=0.75$ (center), and $s=0.95$ (right). In the right panel, the dashed line depicts the condition numbers for the corresponding discrete local problems.} 	\label{fig:conditioning-non-coercive}
 \end{center}
\end{figure}

\subsubsection{Energy comparison in the limit $s \to 1$} \label{sec:energy-comparison}
Our final experiment aims to compare the two energies \eqref{eq:energy-I} and \eqref{eq:energy-III}, and highlight the different fashion in which both recover the limiting local energy as $s \to 1$.
On the square $\Omega = (-1/2,1/2)^2$, we consider $u_{loc}(x_1,x_2) = (1/4 - x_1^2) (1/4 - x_2^2)$, which solves $-\Delta u_{loc} = f$  with homogeneous Dirichlet boundary conditions and $f(x_1,x_2) = -2 (x_1^2 + x_2^2) + 1$. Considering $\Ol = \Omega \cap \{ x_1 < 0 \}$, Figure \ref{fig:energy-comparison} depicts the minimizers $u_I$ and $u_{II}$ we obtained resp. for the two energies $E_I$ and $E_{II}$ for $s = 0.999$, as well as the corresponding solution to the local problem. While visually both solutions to the coupled local/nonlocal models seem to approximate the desired solution, it seems the minimizer of $E_{II}$ is closer to it, while the coupling for the minimizer of \eqref{eq:energy-I} is weaker. 
To make this statement more precise, Table \ref{tab:energy-comparison} compares the $L^2(\Omega)$-norms of such discrepancies for values of $s$ close to $1$. While for the energy $E_{II}$ the convergence seems to be linear with respect to $(1-s)$, the behavior for $E_I$ is less clear and the discrepancies are of a larger magnitude. 

 \begin{figure}[h!] 
\begin{center} \begin{tabular}{ccccc}
 \includegraphics[width=.325\textwidth,keepaspectratio=true]{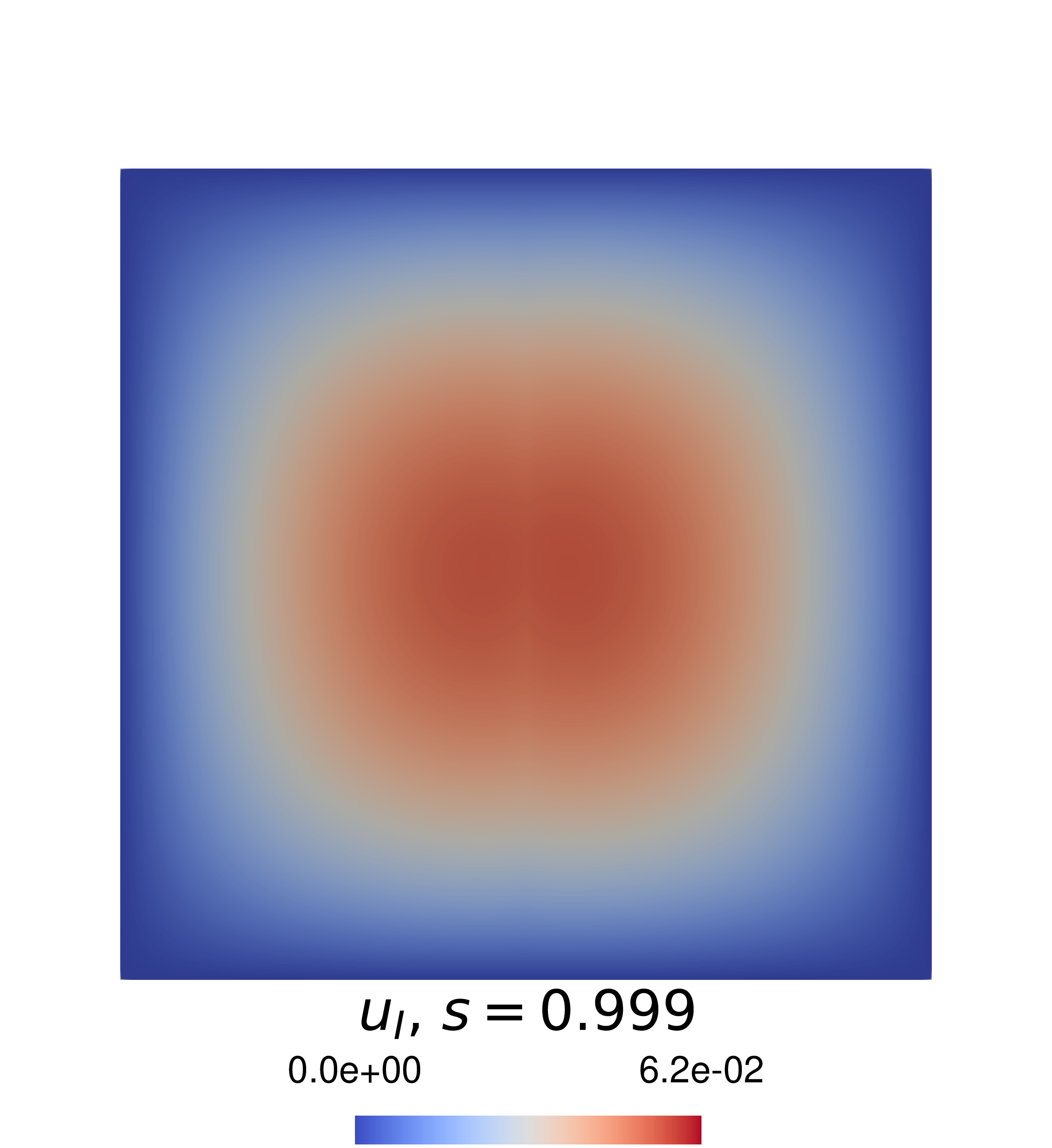} & \hspace{-0.75cm}  & \includegraphics[width=.325\textwidth,keepaspectratio=true]{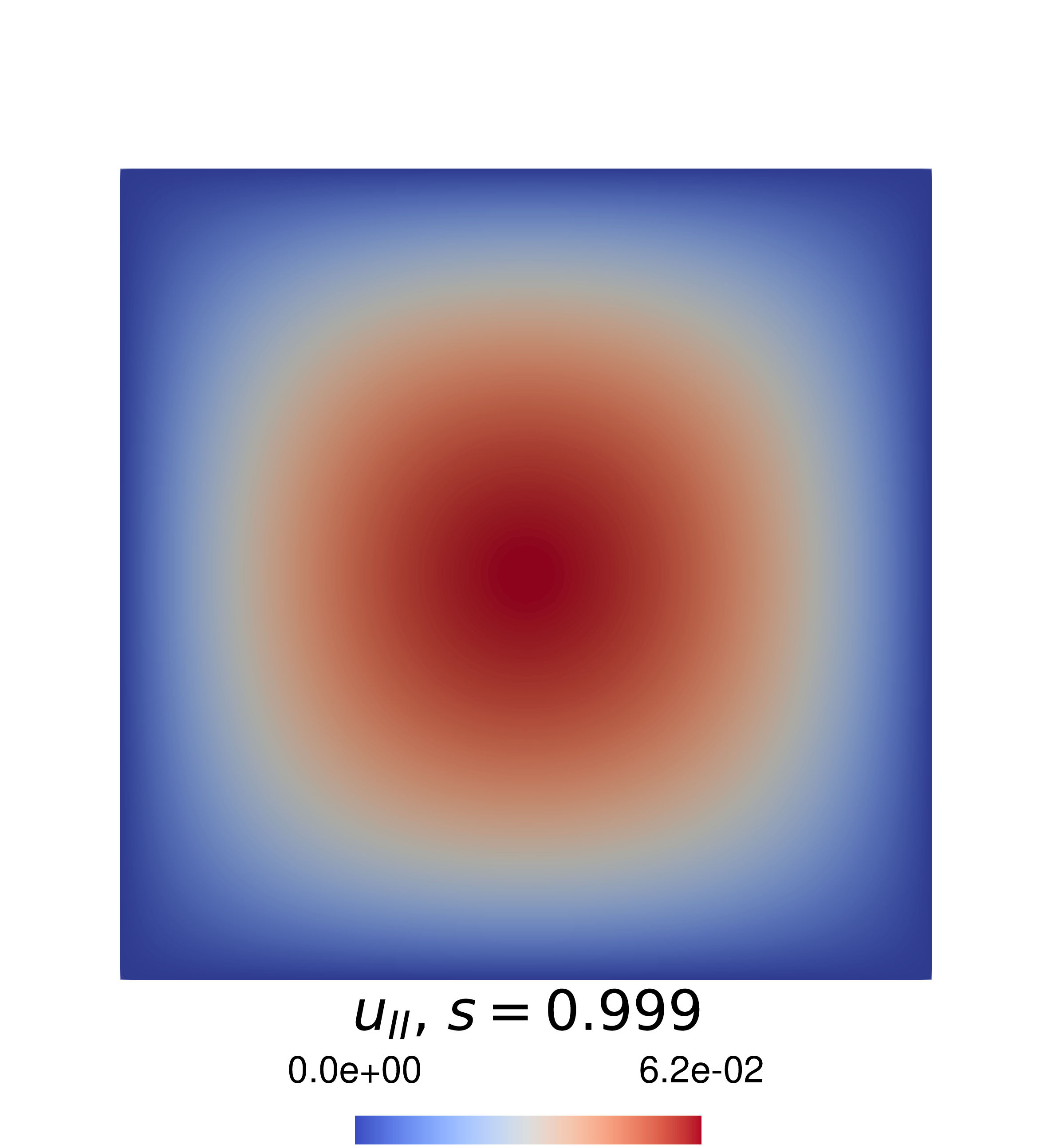}& \hspace{-0.75cm}   & 
   \includegraphics[width=.325\textwidth,keepaspectratio=true]{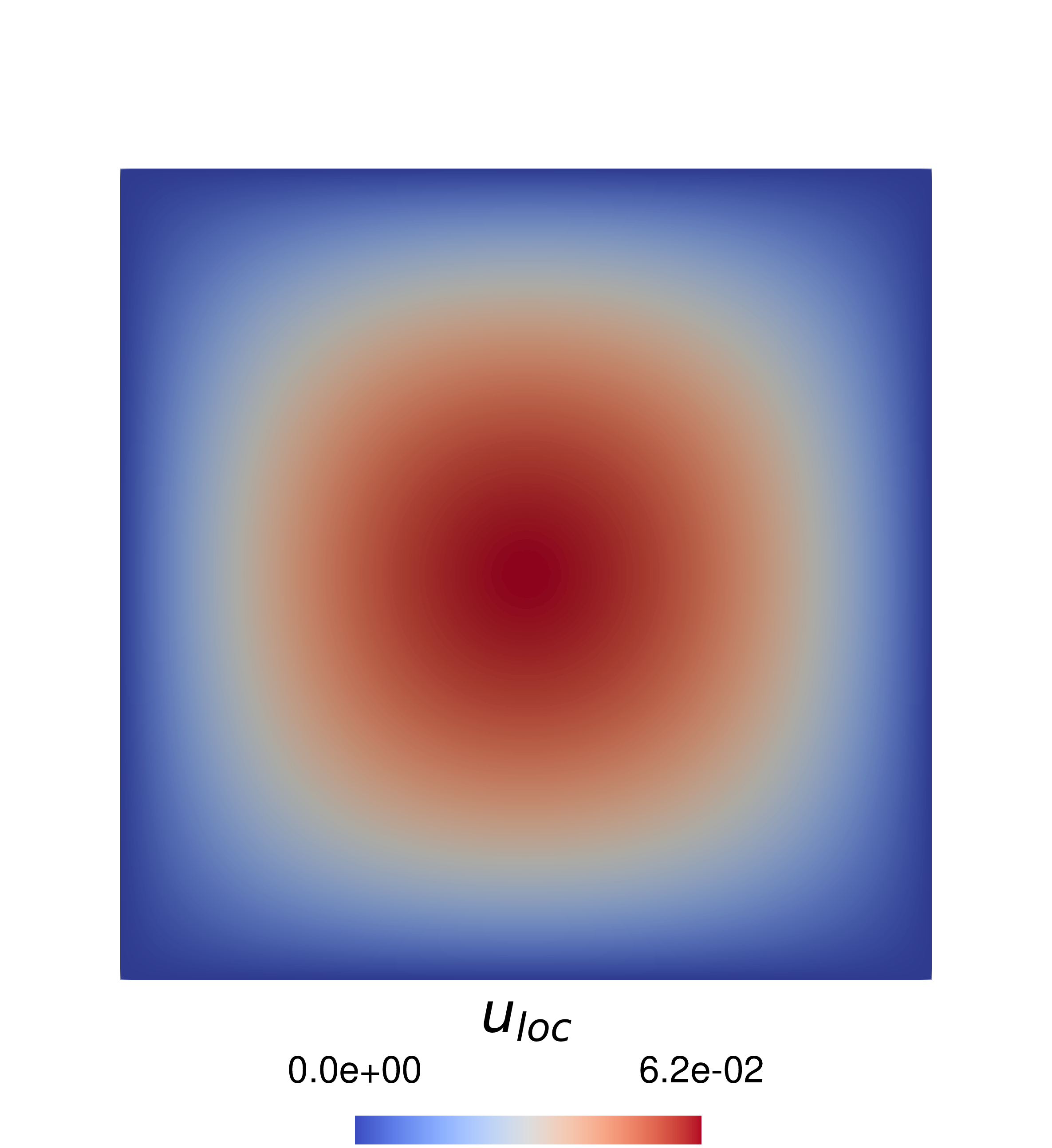} 
      \end{tabular}
	\caption{Discrete energy minimizers for $s=0.999$ (left and center, corresponding to $E_I$ and $E_{II}$, respectively), and solution to the limiting local problem. We use the same scale for the three plots.} 	\label{fig:energy-comparison}
 \end{center}
\end{figure}

\begin{table}[h!]
\begin{center}
\begin{tabular}{c|c|c|c|} 
\cline{2-4}
 & $s = 0.9$  &  $s = 0.99$ & $s = 0.999 $\\ \hline
\multicolumn{1}{|c|}{$E_I$} & $1.6 \times 10^{-2}$ & $1.3 \times 10^{-2}$ & $3.2 \times 10^{-3}$ \\
\hline
\multicolumn{1}{|c|}{$E_{II}$} & $7.1 \times 10^{-3}$ & $6.4 \times 10^{-4}$ & $6.4 \times 10^{-5}$ \\
\hline
\end{tabular}
	\caption{Discrepancy between solutions to \eqref{eq:Euler-Lagrange-E_I-discrete} and \eqref{eq:weak-EIIv3-discrete}  for different values of $s$ and $u_{loc}$  in the $L^2(\Omega)$-norm.} 	\label{tab:energy-comparison}
 \end{center}
\end{table}

\appendix
\section{Some basic results on Besov spaces} \label{sec:Besov}
This appendix collects a few well-known results about Besov spaces that we use in the paper. While it is customary to introduce such spaces by real interpolation between integer-order Sobolev spaces, here we employ an equivalent definition through difference quotients. In particular, we restrict to spaces with differentiability order between 0 and 1. Given a function $v \colon \R^n \to \R$ and a vector $h \in \R^n \setminus \{0\}$, we denote by $\tau_hv(x) := v(x+h)$ the translation of $v$ with vector $h$. Given $\domain \subset \R^n$ and $\rho > 0$, we write
\[
\domain_\rho := \{ x \in \domain \colon d(x, \partial\domain)<\rho \}.
\]

\begin{definition} Let $D$ be a fixed ball centered at the origin, $\delta \in (0,1)$, $p,q \in [1,\infty]$. Then, 
\begin{equation} \label{eq:def-Besov}
|v|_{B^\delta_{p,q}(\domain)} := \Big(q\delta(1-\delta)
\int_{D} \frac{\| \tau_h v - v  \|_{L^p(\domain_{|h|})}^q}{|h|^{n+q\delta}}dh \Big)^{1/q}
\end{equation}
 for $p,q\in[1,\infty)$ while for $q=\infty$ we let
\[
|v|_{B^\delta_{p,\infty}(\domain)} := \sup_{h \in D} 
\frac{\| \tau_h v -  v \|_{L^p(\domain_{|h|})}}{|h|^{\delta}},
\]
\end{definition}

Importantly, when $p = 2$ we have the algebraic and topological equivalence $H^\delta(\domain) = B^\delta_{2,2}(\domain)$ for all $\delta \in \R$, and actually the same equivalence between Besov and Sobolev spaces holds for an arbitrary $p \in [1,\infty]$ as long as the differentiability order $\delta$ is not an integer.

We also make use of Besov spaces of functions that are supported on domains. For that purpose, we define
\[
\dot{B}^\delta_{p,q}(\domain) := \{ v \in B^\delta_{p,q}(\R^n) \colon \mbox{supp } v \subset \overline \domain \}.
\]
The Besov scale on bounded Lipschitz domains possesses some relevant orderings with respect to its defining parameters. While the inclusion
\[ \begin{array}{llll}
B^{\delta_1}_{p,q_1}(\domain) \subset B^{\delta_0}_{p,q_0}(\domain) & \mbox{ if } 0 < \delta_0 < \delta_1, &   1 \le p \le \infty , & 1 \le q_0, q_1 \le \infty,
\end{array} \]
is well-known (for example, \cite[\S3.3.1]{Triebel10}), we are interested in a quantitative estimate.
The following result can be proven by repeating exactly the same steps as in \cite[Lemma 2.1]{BoLiNo23}.

\begin{proposition}
Let $\Omega \subset \R^n$ be a bounded Lipschitz domain, $p,q \in [1,\infty)$, $\delta\in (0,1)$, and $\epsilon \in (0,1-\delta)$. Then, $B^{\delta+\epsilon}_{p,\infty}(\Omega) \subset B^{\delta}_{p,q} (\Omega)$ with 
\begin{equation}\label{eq:Besov-Sobolev-emb}
\|v\|_{ B^{\delta}_{p,q}(\Omega)} \lesssim 
\left( \frac{\delta}{\epsilon} \right)^{\frac1q} \|v\|_{B^{\delta+\epsilon}_{p,\infty}(\Omega)}  \quad\forall \, v\in B^{\delta+\epsilon}_{p,\infty}(\Omega).
\end{equation}
\end{proposition}

Finally, we recall another well-known fact: that Besov spaces are increasing with respect to the parameter $q$. This follows from the characterization of their seminorms as discrete sums,
\[
| v |_{B^\delta_{p,q}(\domain)} \simeq
\left\lbrace
\begin{aligned}
\left( \sum_{k=0}^\infty | 2^{k\delta} \sup_{|h| \le 2^{-k}} \| \tau_h v - v \|_{L^p(\domain_{|h|})}|^q \right)^{1/q} & \quad \mbox{ if } q < \infty, \\
\sup_{k \ge 0} \, 2^{k\delta} \sup_{|h| \le 2^{-k}} \| \tau_h v - v \|_{L^p(\domain_{|h|})} & \quad \mbox{ if } q = \infty ,
\end{aligned} 
\right.
\]
and the fact that $\ell^q$ spaces are increasing with respect to $q$. We refer to \cite[\S2.10]{DVLo} for details.

\begin{lemma}
Let $\Omega \subset \R^n$ be a bounded Lipschitz domain, $p, \in [1,\infty)$, $1 \le q \le q' \le \infty$, and $\delta\in (0,1)$. Then, $B^{\delta}_{p,q}(\Omega) \subset B^{\delta}_{p,q'} (\Omega)$ with 
\begin{equation}\label{eq:Besov-Besov-emb}
\|v\|_{ B^{\delta}_{p,q'}(\Omega)} \lesssim  \|v\|_{B^{\delta}_{p,q}(\Omega)}  \quad\forall \, v\in B^{\delta}_{p,q}(\Omega).
\end{equation}
\end{lemma}

\section{Proof of Lemma \ref{lemma:mapping-Laplacian}} \label{sec:mapping-Laplacian}
In this appendix, we prove Lemma \ref{lemma:mapping-Laplacian} about the mapping properties of the weighted fractional Laplacian \eqref{eq:def-wifl}.
We note that, because $0<\alpha < 1-2s$, according to Definition \eqref{eq:def-Besov}, it suffices to show that 
\begin{equation} \label{eq:mapping-Laplacian}
\| (\tau_h - I) (-\Delta_\snl)^s u \|_{L^2(\R^n)} \lesssim |h|^{\alpha}  \| u \|_{H^{\alpha+2s}(\R^n)}.
\end{equation}
For that purpose, given $x \in \R^n$ we split
\small{\begin{equation} \label{eq:split-translation-Laplacian}
\begin{split}
& (\tau_h - I) (-\Delta_\snl)^s u (x) = (-\Delta_\snl)^s u (x+h) - (-\Delta_\snl)^s u (x) \\
& =  \! \int_{\R^n} ( \snl(x+h, y+h) - \snl(x+h, x+h) ) \left( \frac{u(x+h) - u(y+h) - u(x) + u(y)}{|x-y|^{n+2s}} \right)  dy \\ 
& +  \int_{\R^n} \snl(x+h, x+h) \left( \frac{u(x+h) - u(y+h) - u(x) + u(y)}{|x-y|^{n+2s}} \right) dy \\ 
& +  \int_{\R^n} \left( \snl(x+h, y+h) -  \snl(x, y) \right) \left( \frac{u(x) - u(y)}{|x-y|^{n+2s}} \right) dy \\ 
& =: I(x) + II(x) + III(x) .
\end{split}
\end{equation}}

We begin bounding the term $I$. We exploit the boundedness and H\"older regularity of $\snl$, cf. \eqref{eq:Holder-sigma}, and fix some $\theta < \min\{2s, \beta\}$, to deduce 
\[
|\snl(x+h, y+h) -  \snl(x+h, x+h) | \lesssim |x-y|^\theta.
\]
With that, making the change of variables $z = y - x$ and using Minkowski's inequality,
\[ \begin{split}
\| I \|_{L^2(\R^n)} & \lesssim \left( \int_{\R^n} \left(\int_{\R^n} \frac{|u(x+h)-u(y+h) -u(x) + u(y)|}{|x-y|^{n+2s-\theta}} \, dy \right)^2 dx \right)^{1/2} \\
& = \left( \int_{\R^n} \left(\int_{\R^n} \frac{|u(x+h)-u(x+h+z) -u(x) + u(x+z)|}{|z|^{n+2s-\theta}}\, dz \right)^2 dx \right)^{1/2} \\
& \le \int_{\R^n} \frac{\| \tau_z( u - \tau_h u) -  ( u - \tau_h u)\|_{L^2(\R^n)}}{|z|^{n+2s-\theta}} \, dz \\
& \lesssim \| u - \tau_h u \|_{B^{2s-\theta}_{2,1}(\R^n)}.
\end{split} \]
Next, we observe that, by interpolation of the inequalities
\[
 \| u - \tau_h u \|_{L^2(\R^n)} \lesssim |h|^{\alpha+2s} \| u \|_{H^{\alpha+2s}(\R^n)} , \qquad  \| u - \tau_h u \|_{H^{\alpha+2s}(\R^n)} \lesssim \| u \|_{H^{\alpha+2s}(\R^n)} ,
\]
we have
\[
 \| u - \tau_h u \|_{B^{2s-\theta}_{2,1}(\R^n)} \lesssim |h|^{\alpha + \theta} \| u \|_{H^{\alpha+2s}(\R^n)}.
\]
Therefore, since $|h| \le 1$, we obtain
\begin{equation} \label{eq:bound-I}
\| I \|_{L^2(\R^n)} \lesssim |h|^{\alpha + \theta} \| u \|_{H^{\alpha+2s}(\R^n)} \le  |h|^{\alpha} \| u \|_{H^{\alpha+2s}(\R^n)} .
\end{equation}

We now deal with the term $II$ in \eqref{eq:split-translation-Laplacian}. For this, we note that it essentially coincides with a first difference for a non-weighted fractional Laplacian, namely, 
\[
II (x) = \snl(x+h, x+h) \left( \tau_h - I \right) (-\Delta)^s u (x),
\]
and thus
\[
\| II \|_{L^2(\R^n)} \lesssim  \| \left( \tau_h - I \right) (-\Delta)^s u \|_{L^2(\R^n)} \lesssim |h|^{\alpha} \|  (-\Delta)^s u \|_{B^{\alpha}_{2,\infty}(\R^n)}.
\]
We use the continuity of the embedding $H^{\alpha}(\R^n) \subset B^{\alpha}_{2,\infty}(\R^n)$ (see \eqref{eq:Besov-Besov-emb}) and the well-known mapping property $(-\Delta)^s \colon H^{\alpha+2s}(\R^n) \to H^{\alpha}(\R^n)$ to obtain
\begin{equation} \label{eq:bound-II}
\| II \|_{L^2(\R^n)} \lesssim |h|^{\alpha} \|  u \|_{H^{\alpha+2s}(\R^n)}.
\end{equation}

Finally, we bound the integral $III$ in \eqref{eq:split-translation-Laplacian}. We again exploit the $\beta$-H\"older regularity of $\snl$ \eqref{eq:Holder-sigma},
\[
\| III \|_{L^2(\R^n)}  \lesssim |h|^\beta \left( \int_{\R^n} \left( \int_{\R^n} \frac{|u(x) - u(y)|}{|x-y|^{n+2s}} \, dy \right)^2 dx \right)^{1/2} .
\]
Upon making the change of variables $z = x + y$ and using Minkowski's inequality, we obtain
\[
 \left( \int_{\R^n} \left( \int_{\R^n} \frac{|u(x) - u(y)|}{|x-y|^{n+2s}} \, dy \right)^2 dx \right)^{1/2} = \int_{\R^n} \frac{\| \tau_z u - u \|_{L^2(\R^n)}}{|z|^{n+2s}} \, dz
\] 
and therefore, by definition \eqref{eq:def-Besov},
\begin{equation} \label{eq:bound-III}
\| III \|_{L^2(\R^n)} \lesssim |h|^{\beta} \|  u \|_{B^{2s}_{2,1}(\R^n)}.
\end{equation}

The bound \eqref{eq:mapping-Laplacian} follows by putting together \eqref{eq:bound-I}, \eqref{eq:bound-II}, and \eqref{eq:bound-III} and noticing that $\|  u \|_{B^{2s}_{2,1}(\R^n)} \lesssim \|  u \|_{H^{\alpha+2s}(\R^n)}$ by \eqref{eq:Besov-Sobolev-emb}-\eqref{eq:Besov-Besov-emb}.


\bibliographystyle{abbrv}
\bibliography{coupling}

\end{document}